%% file: main.tex
\documentclass{article}
\usepackage[a4paper, margin=3cm
 ]{geometry}
\usepackage[utf8]{inputenc}
\usepackage{mathpazo}
\usepackage[hidelinks]{hyperref}
\usepackage[anythingbreaks]{breakurl}
\usepackage{doi}
\usepackage{relsize}
\usepackage{fontawesome}

\usepackage[T1]{fontenc}
\usepackage{amssymb, amsmath}
\usepackage[dvipsnames]{xcolor}
\usepackage[english]{babel}
\usepackage{amsthm}
\usepackage{enumitem}
\usepackage{tikz}
\usetikzlibrary {arrows.meta}
\usepackage{bm}
\usepackage{csquotes}
\usepackage[shortcuts]{extdash}
\usepackage[capitalise]{cleveref}
\usepackage{xcolor}
\usepackage[colorinlistoftodos]{todonotes}

\usepackage[format=plain,
            labelfont=it,
            textfont=it, singlelinecheck=false]{caption}
\usepgflibrary{shapes.geometric}
\usepackage{float}
\allowdisplaybreaks

 \hypersetup{breaklinks=true, 
colorlinks=false,   
        }

\usepackage[maxnames=20, url=false, doi=false]{biblatex} 
\bibliography{main.bib}
\AtEveryBibitem{\clearfield{issn}}

\usepackage{subfiles}

\AtBeginBibliography{\small}
 
\newtheorem{theorem}{Theorem}[section]

\newtheorem{corollary}[theorem]{Corollary}

\newcommand{\ignore}[1]{}

\newtheorem{lemma}[theorem]{Lemma}
\newtheorem{prop}[theorem]{Proposition}
\newtheorem{fact}[theorem]{Fact}
\theoremstyle{remark}
\newtheorem{remark}[theorem]{Remark}
\newtheorem{notation}[theorem]{Notation}

\newtheorem{observation}[theorem]{Observation}

\theoremstyle{definition}

\newtheorem{examples}[theorem]{Examples}

\newtheorem{definition}[theorem]{Definition}

\newtheorem{conjecture}[theorem]{Conjecture}

\def\acts{\curvearrowright}

\input{commands}

\usepackage{orcidlink}

\title
{All mixed identities are singular\\ in groups with no algebraicity}

\author{Paolo Marimon\thanks{%
Department of Computer Science, University of Oxford. Oxford, United Kingdom.\\
Email: \href{mailto:paolo.marimon@cs.ox.ac.uk}{paolo.marimon@cs.ox.ac.uk}\\
ORCID: \href{https://orcid.org/0000-0001-9355-1685}{0000-0001-9355-1685} \orcidlink{0000-0001-9355-1685}
}\and Michael Pinsker\thanks{%
Institut f\"ur Diskrete Mathematik \& Geometrie. Technische Universit\"{a}t Wien, Vienna, Austria.\\
Email: \href{mailto:michael.pinsker@tuwien.ac.at}{michael.pinsker@tuwien.ac.at}\\
ORCID: \href{https://orcid.org/0000-0002-4727-918X}{0000-0002-4727-918X} \orcidlink{0000-0002-4727-918X}}}

\date{}

\begin{document}
\maketitle
\begin{abstract}
We show that if {a group}  admits an action with no algebraicity 
then all of its mixed identities are singular. Previously, such groups were only known to be lawless by a theorem of Ab\'{e}rt. 

Our  result 
confirms, in particular, a conjecture of Bodirsky, Schneider, and Thom for a large class of oligomorphic permutation groups. It thereby not only subsumes
numerous results from the literature 
in a simple uniform theorem, but also settles the question for prominent groups for which the conjecture was an open problem, such as the automorphism group of $(\mathbb{Q}; <)$. It also  applies outside the oligomorphic context, e.g.~to much-investigated groups such as  
 Thompson's groups $F, T$, and $V$, 
to Grigorchuk’s group, and to the homeomorphism groups of any manifold of dimension $\geq 1$. 


More generally, we prove that all mixed identities of {a group}  are singular as long as it has an action satisfying certain geometric conditions. This additionally covers, for example, the infinite-dimensional general and projective linear groups. 

\end{abstract}

\input{Sections/Introduction}

\input{Sections/Preliminaries}

\input{Sections/mixed}

\input{Sections/Maintheorem}
\input{Sections/Examples}
\input{Sections/Microsupported}

\textbf{Acknowledgment:} We would like to thank Luna Elliott, Itay Kaplan, and Oria Frenkel for helpful discussions. This work was supported by UKRI EP/X024431/1. Funded by the European Union (ERC, POCOCOP, 101071674). Views and opinions expressed are however those of the authors only and do not necessarily reflect those of the European Union or the European Research Council Executive Agency. Neither the European Union nor the granting authority can be held responsible for them. This research was funded in whole or in part by the Austrian Science Fund (FWF) [I 5948]. For the purpose of Open Access, the authors have applied a CC BY public copyright licence to any Author Accepted Manuscript (AAM) version arising from this submission.

 \printbibliography

\end{document}

%% file: commands.tex
 \definecolor{Oxblue}{cmyk}{0.96,0.75,0,0}   

 \definecolor{Oxred}{cmyk}{0.18, 1, 0.74, 0.08}

\definecolor{Oxsky}{cmyk}{0.34,0.06,0,0}

\definecolor{Oxyellow}{cmyk}{0.1, 0.23, 0.93, 0.01}
\definecolor{Oxlemon}{cmyk}{0.08, 0, 0.69, 0}
\definecolor{Oxvivid}{cmyk}{0.56, 0, 0.46, 0}

\definecolor{Oxlime}{cmyk}{0.54, 0, 1, 0}

\definecolor{Oxgray}{cmyk}{0.06,0.06,0.09,0}

\definecolor{TUblue}{cmyk}{1,0.38,0,0.15}
\definecolor{Oxmauve}{cmyk}{0.58, 0.6, 0.27, 0.1}

\definecolor{Oxcool}{cmyk}{0.15, 0, 0.08, 0}
\definecolor{Oxgreen}{cmyk}{0.79, 0.35, 0.64, 0.26}

\newcommand{\indep}[2]{%
  \mathrel{
    \mathop{
      \vcenter{
        \hbox{\oalign{\noalign{\kern-.3ex}\hfil$\vert$\rlap{$^\mathrm{#2}$}\hfil\cr
              \noalign{\kern-.7ex}
              $\smile$\cr\noalign{\kern-.3ex}}}
      }
    }\displaylimits_{#1}
  }
}

\newcommand{\notindep}[2]{%
  \mathrel{
    \mathop{
      \vcenter{
        \hbox{\oalign{\noalign{\kern-.3ex}\hfil$\not\vert$\rlap{$^\mathrm{#2}$}\hfil\cr
              \noalign{\kern-.7ex}
              $\ \smile$\cr\noalign{\kern-0.3ex}}}
      }
    }\displaylimits_{#1}
  }
}

\DeclareMathOperator{\rk}{rk}
\DeclareMathOperator{\cl}{cl}
\DeclareMathOperator{\pcl}{pcl}
\DeclareMathOperator{\acl}{acl}

\DeclareMathOperator{\Aut}{Aut}

\DeclareMathOperator{\Homeo}{Homeo}
\DeclareMathOperator{\Diff}{Diff}

\DeclareMathOperator{\Sym}{Sym}
\DeclareMathOperator{\GL}{GL}

\DeclareMathOperator{\PGL}{PGL}
\DeclareMathOperator{\Alt}{Alt}

\DeclareMathOperator{\ZZ}{Z}
\newcommand{\zar}{\tau_{\ZZ}}
\DeclareMathOperator{\ppww}{pw}
\newcommand{\pw}{\tau_{\ppww}}

\newcommand{\spOm}{\bm{\Omega}}

\newcommand{\fps}[1]{[#1]^{<\omega}}

\def\Ind#1#2{#1\setbox0=\hbox{$#1x$}\kern\wd0\hbox to 0pt{\hss$#1\mid$\hss}
	\lower.9\ht0\hbox to 0pt{\hss$#1\smile$\hss}\kern\wd0}
\def\ind{\mathop{\mathpalette\Ind{}}}
\def\Notind#1#2{#1\setbox0=\hbox{$#1x$}\kern\wd0\hbox to 0pt{\mathchardef
		\nn="3236\hss$#1\nn$\kern1.4\wd0\hss}\hbox to
	0pt{\hss$#1\mid$\hss}\lower.9\ht0
	\hbox to 0pt{\hss$#1\smile$\hss}\kern\wd0}
\def\nind{\mathop{\mathpalette\Notind{}}}

\newenvironment{subproof}[1][\proofname]{%
  \begin{proof}[#1]%
}{%
  \end{proof}%
}

 \newtheorem{thmx}{Theorem}

%% file: Sections/Introduction.tex
\section{Introduction}\label{sect:intro}

One of the very core problems of algebra is understanding the solutions to equations  over  algebraic structures; the powerful tools of Galois theory, for example, are among the most fundamental and prominent achievements 
 of this kind
in  the case of fields. Over a group $G$, equations can be written in  the form $w(x_1, \dots, x_r)=1$, where $w(x_1, \dots, x_r)$ is a word with constants from $G$ and variables $x_1, \dots, x_r$. 
On one hand, we 
wish to understand when $w$ is solvable in $G$, that is, when we can find $g_1, \dots, g_r\in G$ such that $w(g_1, \dots, g_r)=1$. On the other hand, we wonder 
when $w$ is a \textbf{mixed identity} of $G$ (called a \textbf{law} if it has no constants): that is, for all $g_1, \dots, g_r\in G$, $w(g_1, \dots, g_r)=1$. There is an obvious tension between solving equations and finding mixed identities: if $w$ is a mixed identity of $G$, then for any $\gamma\in G\setminus\{1\}$, $\gamma w$ is not solvable in $G$. 

A word is \textbf{singular} if, when ``forgetting'' its constants (by replacing them by $1$) and reducing,  one obtains the identity: for example, this is the case for the word $x_1^{-1}\gamma x_1$,  where $\gamma \in G$. We call non-singular words \textbf{regular}. Clearly, in the case of laws (without constants) only regular words are of interest.  Mixed identities, on the other hand,  do  occur commonly with singular words even in infinite groups, but seem to be rare with  regular words. This observation is underpinned by the fact that while there are singular words over groups $G$ which have no solution in any overgroup $H\supseteq G$~\cite[$\S$1]{klyachko2017new}, it is a major open problem whether this phenomenon might occur also for regular words~\cite[Conjecture 1.1]{klyachko2017new}: given a regular word with constants from a group $G$, can we, in analogy to fields, always find a solution to it in an adequate extension?
The one-variable case of this question is the famous Kervaire-Laudenbach Conjecture, answered positively for finite groups~\cite{gerstenhaber1962solution} and for hyperlinear groups~\cite{pestov2008hyperlinear}.


In order to address the problem of equation solvability, it is of obvious  importance to isolate general conditions under which regular word maps (Definition~\ref{def:word}) have large, ideally surjective, image;  dually, 
this calls for versatile techniques and 
 general
assumptions that allow us to derive the absence of mixed identities in a group. 
Mixed identities have been studied intensively over the past 50 years, starting with the work of Soviet mathematicians~\cite{anashin1978mixed, tomanov1985generalized, golubchik1982generalized}, and with substantial progress in the last decade~\cite{hull2016transitivity, bradford2023non, bradford2025length, etedadialiabadi2021dense, ivanov2025mixed}. Some of the most impressive results on mixed identities provide  
conditions under which a given class of groups is \textbf{mixed-identity-free (MIF)} (i.e., with no non-trivial mixed identities), often with a dichotomy theorem drawing 
a dividing line between groups in a given class which are MIF and those which are not~\cite{tomanov1985generalized, hull2016transitivity, bradford2023non}. For example, the 
Hull-Osin dichotomy~\cite{hull2016transitivity} shows that a highly transitive countable group is MIF if and only if it does not have a normal subgroup isomorphic to $\Alt(\mathbb{N})$. However, the mixed identities implied by the presence of $\Alt(\mathbb{N})$ are all singular (cf.~\cite{bodirsky2024mixed} and Theorem~\ref{mainthm:noalg}). This invites us again to 
develop general techniques to prove that all mixed identities are singular.

Recently, Bodirsky, Schneider, and Thom~\cite{bodirsky2024mixed} 
{proposed}
a natural richness condition for a group under which they conjecture that all mixed identities are singular: a {group action}
$G\acts\Omega$, where $\Omega$ is a countably infinite set, is \textbf{oligomorphic} if the diagonal action $G\acts \Omega^n$ where $g\cdot (a_1, \dots, a_n)=(g a_1, \dots, g a_n)$ has finitely many orbits for all $n\geq 1$~\cite{cameron2009oligomorphic}. Examples of oligomorphic permutation groups are $\Sym(\mathbb{N})$, all countable highly transitive groups, the automorphism group of the rational numbers with their order $\Aut(\mathbb{Q}; <)$, general linear groups of countably infinite vector spaces  
over finite fields~\cite{evans1991small}, and the automorphism groups of homogeneous structures in a finite relational language~\cite{homogeneous}. These groups are of special interest to model theory since they appear as automorphism groups of countable $\omega$-categorical structures. 
We know from~\cite{macpherson1986groups} that oligomorphic permutation groups are lawless (i.e., with no non-trivial law), and in various important contexts we have general 
conditions
 under which word maps without constants are surjective~\cite{adeleke1994representation, mycielski1987representations, lyndon1990words, maroli1990representation}. {In recent years, mixed identities in oligomorphic permutation groups have} received considerable attention~\cite{hull2016transitivity, etedadialiabadi2021dense, ghadernezhad2019group, bradford2023non, bodirsky2024mixed} with several cases 
of such groups being shown to be MIF, or such that all of their mixed identities are singular. This  prompted~\cite{bodirsky2024mixed} to conjecture:
\begin{conjecture}[{Conjecture 1 in~\cite{bodirsky2024mixed}}]\label{conj1} Let $G$ be a group with an oligomorphic action $G\acts \Omega$. Then all mixed identities of $G$ are singular.
\end{conjecture}
Despite the abundance of examples, there are substantial challenges to Conjecture~\ref{conj1}: existing techniques to show that a given word is not a mixed identity usually rely on 
the constants that might appear in a word satisfying particular tameness conditions which are often 
so strong that they imply MIF (cf.~\cite{bodirsky2024mixed}); {but many oligomorphic permutation groups,} {including $\Sym(\mathbb{N})$ and $\mathrm{Aut}(\mathbb{Q}; <)$,  are known not to be MIF}. In particular, in spite of considerable work being dedicated to it in~\cite{bodirsky2024mixed}, the conjecture was hitherto open even for the fundamental case of $\Aut(\mathbb{Q}; <)$.

In the present paper, we find one single, simple, {and 
general} richness condition on a group action that implies that all mixed identities of the group are indeed singular, and prove this by methods that are uniform in the sense that they do not depend on an analysis of the possible behaviour of constants (which at this level of generality 
{seems unfeasible}). 
To our knowledge, there has been no prior result of comparable scope ruling {out}
regular mixed identities. 
It not only confirms Conjecture~\ref{conj1} for a large class of oligomorphic permutation groups including $\Aut(\mathbb{Q};<)$, but also applies to numerous  geometric and topological examples beyond the oligomorphic context. We say that a 
{group action} $G\acts\Omega$ has \textbf{no algebraicity} (or {$G\acts\Omega$} 
\textbf{separates}  
$\Omega$ in the topology literature~\cite{abert2005group, ivanov2025mixed})
if for any finite $B\subseteq\Omega$, every $a\in\Omega\setminus B$ has an infinite $G_B$-orbit, where $G_B$ is the pointwise  stabiliser of $B$. 

\begin{thmx}[Corollary~\ref{cor:main}]\label{mainthm:noalg} {Let $G$ be a group, and let $w$ be a regular word  with constants in $G$. If $G$ has a group action
$G\acts \Omega$  
with no algebraicity, then  the set} 
\[\{(g_1, \dots, g_r)\in G^r\ \vert \ w(g_1, \dots, g_r)\neq 1\}\]
is 
 {dense} in  $G^r$ with respect to the topology of pointwise convergence {induced by the action on the discrete set $\Omega$}. 
In particular, all mixed identities of $G$ are 
then 
singular.
\end{thmx}
We remark that previously, it was {observed by} 
Ab\'{e}rt~\cite{abert2005group} that permutation groups with no algebraicity are lawless.
Theorem~\ref{mainthm:noalg} can therefore be viewed as a vast generalization of that result:  the presence of constants in mixed identities  implies the possibility of non-trivial singular subwords, thereby leading to a major leap in the complexity of the problem.

The property of no algebraicity implies that $\Omega$ and $G$ are infinite, but does not bound the cardinality of either. Note also that if $G\acts\Omega$ has no algebraicity, then so do all of its dense {(with respect to pointwise convergence induced on $G$ by the action)} subgroups, and all permutation groups $H$ such that $G\leq H\leq \Sym(\Omega)$ {(if we view $G$ as a permutation group via the action)}. There is a plethora of interesting groups arising in logic, topology, and geometric group theory {which act naturally} with no algebraicity, 
 which we list below. We know that several of these examples have (singular) mixed identities.

\begin{itemize}
\item The automorphism groups of  \textbf{homogeneous structures} in a relational language whose class of finite substructures has the strong amalgamation property  (this is known to be equivalent to no algebraicity~\cite{homogeneous}). If the  relational language is finite, these automorphism groups are oligomorphic, but we do not require this. Such groups include:
\begin{itemize}
\item the majority of examples whose mixed identities are investigated in~\cite{bodirsky2024mixed, ghadernezhad2019group, etedadialiabadi2021dense}, including the automorphism groups of: $(\mathbb{N}; =)$, {all} homogeneous structures with free amalgamation (such as the Rado graph and the generic triangle-free graph), the random poset, the random permutation, and the Urysohn space. With the exception of $\Sym(\mathbb{N})=\Aut(\mathbb{N}; =)$ and non-transitive homogeneous structures with free amalgamation, the above examples are MIF. However, previous techniques cannot prove that Conjecture~\ref{conj1} also holds for
the automorphism groups of the reducts of such structures (i.e., closed supergroups of their automorphism groups), whilst this follows immediately from our result;
\item the automorphism groups of the following structures, where the presence of intervals or tree-like structure was an obstruction to previous techniques:\footnote{All of these structures are described in Examples 2.3.1 and 6.1.2 of~\cite{homogeneous}.} $(\mathbb{Q}; <)$ and its reducts~\cite{cameron1976transitivity}, the homogeneous local order $S(2)$~\cite{cameron1981orbits}, Droste's $2$-homogeneous semilinear orders of negative type~\cite{droste1985structure}, and homogeneous $C$-relations~\cite{adeleke1998relations}. We know that the first three examples have mixed identities (the first by~\cite{bodirsky2024mixed}, and the second two by~\cite{ghadernezhad2019group} and 
Proposition~\ref{prop:elliott}).

\end{itemize}
\item Any highly transitive group (cf.~\cite{hull2016transitivity}), or any highly order-transitive group (cf.~\cite{glass1991highly}).
\item Any group with a \textbf{Rubin action} on an infinite set in the sense of{~\cite[p.158]{belk2025short} (cf.~\cite{rubin1989reconstruction, kim2021structure}).} 

This includes:
\begin{itemize}
\item Thompson's groups $F, T,$ and $V$, 
in their actions, respectively, on $(0,1), S^1,$ and $2^\omega$ (see~\cite{cannon1996introductory}). 
We know from~\cite{brin1985groups} that Thompson's group $F$ is lawless, and from~\cite{zarzycki2010limits} that it has singular mixed identities; its mixed identities were also studied in~\cite{slanina2019laws, ivanov2025mixed}. Indeed, Ab\'{e}rt's proof of lawlessness  for groups with no algebraicity~\cite{abert2005group} is the starting point for the notion of hereditarily separating action that is central to~\cite{ivanov2025mixed}.  Thompson's group $T$ also has mixed identities (by Proposition~\ref{prop:elliott} and~\cite[Theorem 1.2]{elliott2026zariskitopologyhomeomorphismgroups}). The conclusion of Theorem~\ref{mainthm:noalg} for Thompson's groups $F$ and $T$ is a new result. For Thompson's group $V$, one can recover that it is MIF from the proof of~\cite[Theorem 1.3]{elliott2026zariskitopologyhomeomorphismgroups};

\item the homeomorphism groups of the
following spaces: the Cantor set, the Hilbert cube, or any manifold of dimension at least $1$. Again, this result is new for these homeomorphism groups, with the exception of manifolds of dimension $\geq 2$, where~\cite{elliott2026zariskitopologyhomeomorphismgroups} can be seen as showing they are MIF. In particular, our result is new for the homeomorphism groups of the $1$-dimensional manifolds {$\mathbb{R}$ and $S^1$},   
both of which have mixed identities, as one can deduce from~\cite{elliott2026zariskitopologyhomeomorphismgroups} 
and Proposition~\ref{prop:elliott};  
\item any group $G$ such that $\Diff_c^p(X)_0\leq G\leq \Diff^p(X)$, where $X$ is a smooth connected boundaryless manifold and $p\in \mathbb{N}\cup\{\infty\}$. Here,  $\Diff^p(X)$ is the group of $C^p$ diffeomorphisms of $X$ for $p\in\mathbb{N}$ (or the intersection of all such groups if $p=\infty$), and $\Diff_c^p(X)_0$ is the subgroup of compactly supported diffeomorphisms of $X$ that are isotopic to the identity by a compactly supported isotopy (cf.~\cite[Chapter 6, $\S$ 6]{kim2021structure}). The conclusion of Theorem~\ref{mainthm:noalg} is new for such groups.
\end{itemize}
\item Any weakly branch group in its action on the boundary of its associated tree~\cite{grigorchuk2011some} (cf.~\cite{abert2005group, ivanov2025mixed}). This includes Grigorchuk’s first group in its action on the Cantor space $2^\omega$. The latter 
was given in~\cite{jacobson2021mixed} as an example of a lawless amenable group which is not MIF. Again, the conclusion of Theorem~\ref{mainthm:noalg} is new for these groups. 
\item {Any group acting as a subgroup of {the homeomorphism group of $S^1$}  
 by a proximal action in the sense of~\cite[p.895]{le2022triple}. We know that many such groups are not MIF~\cite[Proposition 4.7]{le2022triple}.}
\item The groups $\Homeo_+[0,1]$, in its action on $(0,1)$, and $\Homeo_+(S^1)$ of orientation-preserving homeomorphisms of the unit interval and of the circle. The question of whether all mixed identities are singular for $\Homeo_+[0,1]$ {was} 
posed as an open problem at the end of~\cite{bodirsky2024mixed}. 
\item The wreath product $G\wr H$ of any two groups $G\acts\Omega$, $H\acts\Omega'$ where $G\acts\Omega$ has no algebraicity. Similarly, any direct product $G\times H$ for which either $G$ or $H$ has an action with no algebraicity. We know from~\cite{jacobson2021mixed} that any wreath product  or direct product of non-trivial groups satisfies a singular mixed identity.
\item Any infinite iterated wreath product of non-trivial finite permutation groups (cf.~\cite{bhattacharjee1995ubiquity, abert2005group}).
\end{itemize}
Our main result in its fullest generality,  Theorem~\ref{theorem:main}, implies Theorem~\ref{mainthm:noalg} and has substantially more general hypotheses: namely  {the existence of a $G$-invariant closure operator which forms}
a \textbf{pregeometry} (also known as a matroid) which is \textbf{modular}, and which satisfies a certain generalisation of $\Pi$.~M.~Neumann's lemma~\cite{neumann1976structure} that we call \textbf{higher Neumann}. {This allows us to deduce that all mixed identities are singular for $\GL(\kappa, K)$ and $\PGL(\kappa, K)$, where $\kappa$ is an infinite cardinal and $K$ is a field. {The case of vector spaces has already been studied in~\cite{golubchik1982generalized} and~\cite{bradford2023non}, who show among other things
that all mixed identities are singular for, respectively, $\GL(n, K)$ where $K$ is infinite, and $\GL(\aleph_0, \mathbb{F}_q)$.}
{However, our proof, as it is not tailored to specific groups, uses substantially different methods.} 
Under these more general assumptions we can also deduce that all mixed identities are singular for the automorphism group of any homogeneous graph (these are classified by~\cite{lachlan1980countable}). 

{Several of the topological examples to which Theorem~\ref{theorem:main} applies are groups with micro-supported actions~\cite{caprace2017locally} (sometimes known as locally moving groups~\cite{rubin1996locally}). These are groups with an action on a Hausdorff space where for any open set $V$ we can find some group element stabilising the complement of $V$, but not $V$ (Definition~\ref{def:microsupported}). Micro-supported actions are intensively studied in topological dynamics~\cite{nekrashevych2022groups, caprace2017locally, francoeur2021stabilisers}, starting from Rubin's work~\cite{rubin1989reconstruction, rubin1996locally, mccleary2005locally, ben2010reconstruction}. In fact, Rubin actions are micro-supported (by definition), though in general micro-supported actions may have algebraicity. Groups with micro-supported actions are known to be lawless by~\cite{nekrashevych2022groups}. We lift {this} result to   prove that all of their mixed identities are singular, thereby complementing our main theorem:}
\begin{thmx}[{Theorem~\ref{thm:micro-supported-mixed-identities}}]\label{mainthm:micro} 
Let $G$ be a group and let $w$ be a regular word with constants in $G$. If $G$ has a micro-supported action $G\acts \Omega$ on a Hausdorff topological space $\spOm$, then the set
\[
  \{(g_1,\ldots,g_r)\in G^r\ | \ w(g_1,\ldots,g_r)\neq 1\}
\]
is dense in $G^r$ with the topology of pointwise convergence induced by the action {on $\Omega$ considered as a discrete set.} 
In particular, every mixed identity of $G$ is then  singular.
\end{thmx}


Theorems~\ref{mainthm:noalg} and~\ref{mainthm:micro} also 
have intriguing interactions with the study of intrinsic topologies on groups~\cite{markov1946unconditionally, bryant1977verbal}. In particular, recently there has been substantial activity in describing the behaviour of  Zariski topologies, which are defined in terms of solution sets of disequalities of mixed group or semigroup words (sometimes with restrictions on the allowed words)~\cite{bryant1977verbal, banakh2012algebraically, bardyla2025note, ghadernezhad2019group, pinsker2026zariski, elliott2023automatic,elliott2026zariskitopologyhomeomorphismgroups, gonzalez2026minimal},  with several results focusing precisely on {permutation} groups with no  algebraicity. Indeed, our original motivation to study mixed identities in~\cite{gonzalez2026minimal} arises from this context. More precisely, on a group $G$, the \textbf{group Zariski topology} $\zar$ has a subbasis 
consisting of 
open sets of the form 
\[\{g\in G\ \vert \ w(g)\neq 1\}\;,\]
  where $w$ is a word in a single variable with constants from $G$. Elliott recently observed the following connection between the group Zariski topology and mixed identities, which is a straightforward consequence of~\cite[Lemma 3.1]{elliott2026zariskitopologyhomeomorphismgroups} and~\cite[Proposition 6.18 (i)]{etedadialiabadi2021dense}:
  \begin{prop}[\cite{elliottprivate}]\label{prop:elliott} Let $G$ be a MIF group. Then, the group Zariski topology $\zar$ is \textbf{irreducible}: every non-empty open set is dense. 
  \end{prop}
In particular, if $G$ is MIF, then $\zar$ is not Hausdorff. Hence, we can use Proposition~\ref{prop:elliott} to show that several of the groups within the scope of  {Theorem~\ref{mainthm:noalg}, Theorem~\ref{theorem:main}, and Theorem~\ref{mainthm:micro}}
do satisfy  (necessarily singular) mixed identities: namely, we can resort to existing results in the large literature on Zariski topologies showing that $\zar$ is Hausdorff.

%% file: Sections/Preliminaries.tex
\section{Preliminaries}\label{sect:prelims}

\begin{notation}
{For $n\in\mathbb{N}$, $[n]$ denotes the set $\{0, \dots, n\}$. For a set $C$, we denote by $[C]^{<\omega}$ the set of finite subsets of $C$, and by $\mathcal{P}(C)$ the set of all subsets of $C$.} 
\end{notation}

\subsection{Mixed identities}\label{subsect:mixedidentities}

\begin{definition}\label{def:word} A \textbf{word with constants} $w(x_1, \dots, x_r)$ for the group $G$ and with $r$  variables  is an element
\begin{equation}\label{eq:word}
\gamma_n\ x_{\iota_n}^{\epsilon_n}\ \gamma_{n-1}\  x_{\iota_{n-1}}^{\epsilon_{n-1}}\ \dots\ \gamma_1\ x_{\iota_{1}}^{\epsilon_1}\ \gamma_0 
\end{equation}
of $G*F_r$, 
the free product of $G$ and the free group $F_r$ with $r$ many generators $x_1,\ldots,x_r$ (which we also refer to as \textbf{variables}), where $\iota_1, \dots, \iota_n\in \{1, \dots, r\}$, and $\epsilon_1, \dots, \epsilon_n\in\{-1, 1\}$, and where $\gamma_0,\ldots,\gamma_n$ are elements of $G$. For $0<j< n$, 
we say that $\gamma_j$ is a \textbf{critical constant} if $\iota_{j}=\iota_{j+1}$ and $\epsilon_j=-\epsilon_{j+1}$, 
i.e.~it is sandwiched in $w$ between a variable and its inverse. In order to have unique representations for words, we assume\footnote{{Some of the literature on mixed identities~\cite{bodirsky2024mixed, bradford2023non, bradford2024length, schneider2024word} further requires critical constants to be non-central (and calls such words \textbf{reduced}); 
in other words, $G*F_r$ is factored by the (singular) identities $[x_i,\gamma]=1$ for all $i\in\{1, \dots, r\}$ and $\gamma\in Z(G)$ (which are thereby considered trivial). 
The latter never produces the trivial word from a regular one, and is hence only relevant when studying singular mixed identities.}} 
 that no critical constant is $1\in G$.
 The word $w$ determines the \textbf{word map} $w\colon G^r\to G$, which is obtained by sending $(g_1, \dots, g_r)$ to the image of $w$ under the {group homomorphism $G*F_r\to G$} fixing $G$ and sending $x_i$ to $g_i$ for each {$i\in\{1,\ldots,r\}$}.  The word $w$ is called a \textbf{mixed identity} of $G$ if $w(G^r)=1$.    
 \end{definition}
 \begin{definition}\label{def:singular}
 The \textbf{natural augmentation} is the unique homomorphism  $\rk\colon G*F_r\to F_r$ which sends  each element of $G$ to $1_{F_r}$ and fixes $F_r$. 
  As is common in the literature, we say that $w$ is \textbf{singular} if $\rk(w)=1_{F_r}$; following~\cite{pestov2008hyperlinear}, we say that a word $w$ is  \textbf{regular} otherwise. We call $\rk(w)$ the \textbf{freeness rank} of $w$.   
\end{definition}


\begin{definition}
{We say that $w$ is \textbf{trivial} if $w=1$ in $G*F_r$. 
  The group $G$ is \textbf{mixed-identity-free (MIF)} if its only mixed identity is the trivial word.}
\end{definition}

\subsection{Group actions}

Throughout this paper, $G\acts\Omega$ indicates a 
group $G$ {acting}  {faithfully} 
on an {infinite} set $\Omega$ by permutations. 

\begin{remark}[Non-faithful group actions] \label{rem:unfaithful} {We note that our results also apply to the context of non-faithful actions of $G$, which may be helpful for applications. In that case, we can still endow it with the topology of pointwise convergence defined below. 
Only, the latter will not separate any two group elements acting the same way. {Let $K\unlhd G$ be the kernel of a non-faithful action of $G$ on $\Omega$.}
If that action satisfies the assumptions of either Theorem~\ref{mainthm:noalg} or Theorem~\ref{theorem:main}, then, {for any regular word $w$,} the set $\{(g_1,\ldots,g_r)\in G^r\ \vert \ w(g_1,\ldots,g_r)\not\in K\}$ is dense in $G^r$, 
and 
all mixed identities of $G$ are singular. This is also the case in Theorem~\ref{mainthm:micro} with the caveat that for a non-faithful action of $G$ on the Hausdorff space $\spOm$, we say it is micro-supported if $G_{\Omega\setminus U}\neq K$ for every non-empty open set $U\subseteq\Omega$.}
\end{remark}



\begin{notation} For $G\acts\Omega$, $A\subseteq\Omega$, and $\overline{b}, \overline{b}'$ tuples of elements of $\Omega$ of the same length, we write $\overline{b}\equiv_A\overline{b}'$ if $\overline{b}$ and $\overline{b}'$ are in the same orbit of the componentwise action of $G_A$, the pointwise stabiliser of $A$, on tuples. We extend this notation to tuples  $\overline a$ instead of sets $A$, with the same meaning for the stabiliser of all elements of the tuple.   We write $\overline{b}\equiv\overline{b}'$ to indicate $\overline{b}\equiv_\emptyset\overline{b}'$. For $\Sigma\subseteq G$, we write $\Sigma\cdot \overline b$ for the set $\{\gamma(\overline{b})\;|\; \gamma\in \Sigma\}$; in particular, $G\cdot  \overline{b}$ is the orbit of $\overline b$. 
\end{notation}

\begin{definition}\label{def:top} {For $G\acts\Omega$, we endow $G$ with the \textbf{topology of pointwise convergence} $\pw$, obtained as follows: $\Omega$ is taken to be discrete; $\Omega^\Omega$ is given the product topology; and $G$, viewed as a subset thereof, the induced topology. 
When considering $G^r$ for some $r\in\mathbb{N}$, we shall endow it with the product topology of the topology $\pw$ on $G$.}
For finite tuples $\overline{b}\equiv\overline{b}'$, we write $G_{(\overline{b}, \overline{b'})}$ for the {$\pw$-}open subset of $G$ consisting of all  $h\in G$ such that $h(\overline{b})=\overline{b}'$; these sets 
form a basis of $\pw$. 
\end{definition}

{In Theorem~\ref{mainthm:noalg}, the more general Theorem~\ref{theorem:main}, and Theorem~\ref{mainthm:micro}, we prove that the 
set of tuples of group elements not satisfying an equation given by a regular word is not only non-empty, but even dense with respect to $\pw$.} {This means that   the set of tuples satisfying the equation does not contain any non-empty open set; since it is $\pw$-closed, 
this is tantamount to saying it is nowhere dense. {So, if 
{$G^r$} 
is a Baire space {(for example, when {$G$}  {acts faithfully as a} 
closed permutation group on a countable set),}
the union of solution sets of countably many equations of regular words 
{still has empty interior.}}


%


\begin{definition}\label{def:acl} Let $G\acts\Omega$. We define the operation of \textbf{algebraic closure} $\acl\colon\mathcal{P}(\Omega)\to\mathcal{P}(\Omega)$ as follows:  for $B\in [\Omega]^{<\omega}$ we set   $a\in\acl(B)$ if and only if $G_B\cdot a$ is finite; and for $C\subseteq \Omega$ infinite we set  
\[\acl(C):=\bigcup \{\acl(B)\vert \ B\in\fps{C}\}\;.\]
Note that algebraic closure is a {$G$-{invariant operator}} in the sense that $g\cdot \acl({C})=\acl(g\cdot {C})$ for any $g\in G$ and ${C}\subseteq \Omega$. {We shall discuss such operators  more abstractly in the next section, as our most general theorem (Theorem~\ref{theorem:main}) assumes the existence of such {an} {operator $\cl$ satisfying some general geometric requirements.}}
{We remark that these requirements together imply $\acl(C)\subseteq \cl{(C)}$ for all $C\subseteq \Omega$, cf.~Remark~\ref{rem:minimalityOfacl}.} 

\end{definition}

\begin{definition}
    We say that $G\acts\Omega$ has \textbf{no algebraicity} if for all $B\in[\Omega]^{<\omega}$ we have $\acl(B)=B$.
\end{definition}

It is well-known that if $G\acts\Omega$ is oligomorphic {(cf.~the introduction)}, then $\acl$ is \textbf{locally finite}: for $B$ finite, $\acl(B)$ is finite~\cite[Exercise 4.3.1]{tent2012course}. 

\begin{remark}
    If we defined the  algebraic closure of  infinite sets in the same way as we defined it on finite sets, we would obtain what is sometimes referred to as group-theoretic (or Galois) algebraic closure. The two notions might differ: 
    for example, for $C$ a dense and codense subset of $(\mathbb{Q}; <)$, $\acl(C)=C$, but its group-theoretic algebraic closure is all of $\mathbb{Q}$.
\end{remark}


\subsection{{Closure operators}}

The context of the most general formulation of our result (Theorem~\ref{theorem:main}) will be so that algebraic closure, {or some other abstract $G$-invariant closure operator $\cl$ extending it,}  satisfies  similar axioms to those of linear span in a vector space. For this reason, we need to introduce some standard terminology from matroid theory~\cite{oxley2006matroid} and model theory~\cite{d2023axiomatic}.

\begin{definition} Let $\Omega$ be a set. A \textbf{closure operator} is a function $\cl\colon \mathcal{P}(\Omega)\to\mathcal{P}(\Omega)$ satisfying the following axioms {for all $A,B\subseteq \Omega$}: 

\begin{tabular}{ll}
    {\sc(reflexivity)} & $A\subseteq\cl(A)$;\\
{\sc (transitivity)} & $\cl(\cl(A))=\cl(A)$;\\
     {\sc (monotonicity)} & if $A\subseteq B$, then $\cl(A)\subseteq\cl(B)$.
\end{tabular}

{A} closure operator is {called} \textbf{algebraic} if it 
is of finite character:

\begin{tabular}{ll}
{\sc(finite character)} &  $\cl(A)$ is the union of all $\cl(C)$ where $C$ is a finite subset of $A$.
\end{tabular}

{For an algebraic closure operator $\cl$ on $\Omega$, we say that $(\Omega,\cl)$}   
forms a \textbf{pregeometry} if it satisfies the following 
analogue of Steinitz's exchange  lemma: 

\begin{tabular}{ll}
  {\sc (exchange)} &  if $a\in\cl(A\cup\{b\})\setminus\cl(A)$, then $b\in\cl(A\cup\{a\})$.  
\end{tabular}

     {A set} $A\subseteq \Omega$ is  \textbf{closed} if $\cl(A)=A$. {If $G\acts\Omega$ is a group action such that $\cl(g\cdot C)=g\cdot \cl(C)$ for all $g\in G$ and all $C\subseteq \Omega$, then we call $\cl$ and  $(\Omega,\cl)$ \textbf{$G$-invariant}.}
\end{definition}

Note that for $G\acts\Omega$, the operator $\acl$ satisfies reflexivity, transitivity, and monotonicity; and we have defined it to have finite character. Furthermore, as long as some element of $\Omega$ has an infinite orbit (which will be the case for all groups we consider), $\acl$ is non-trivial in the sense that $\acl(\emptyset)\neq \Omega$. 

In general, exchange may fail to be satisfied by  $\acl$, and hence it does not always 
yield 
a pregeometry. However,  it
is easy to see that $\acl$ does satisfy exchange e.g.~in the following cases: whenever $G\acts\Omega$ has no algebraicity; and  when $G$ is $\mathrm{GL}(\kappa, \mathbb{F}_q)$ or $\mathrm{PGL}(\kappa, \mathbb{F}_q)$, in its action on the corresponding $\kappa$-dimensional vector/projective space where $\kappa$ is an infinite cardinal. {Moreover, linear span in a vector space and projective closure in a projective space over infinite fields are closure operators extending algebraic closure, and also give rise to a pregeometry.} In a pregeometry we can define notions of independent sets, basis, and dimension in the obvious way (see~\cite[Appendix C]{tent2012course} for the claims below). For the reader not familiar with these  notions in this abstract setting, it might be useful to keep the above examples in mind. We remark that in the case of {a group action with} no algebraicity, the notion of dimension {induced by $\acl$} degenerates to that of  cardinality.

\begin{definition} Let $(\Omega, \cl)$ form a pregeometry. We call  $A\subseteq \Omega$ an \textbf{independent set} if for each $b\in A$ we have $b\not\in\cl(A\setminus\{b\})$. We say that $A\subseteq \Omega$ is a \textbf{generating set} if $\Omega=\cl(A)$. A \textbf{basis} 
is an independent generating set. It is easy to prove that every pregeometry has a basis and all bases in a pregeometry have the same cardinality. The \textbf{dimension} of $\Omega$ is the cardinality of any/all of its bases.  
\end{definition}

The above notions can then be extended to arbitrary subsets of a pregeometry in the standard way. For this it will be helpful to define restrictions and relativisations.

\begin{definition} For a pregeometry $(\Omega, \cl)$ and $A\subseteq \Omega$, we can define two new pregeometries:
\begin{itemize}
    \item the \textbf{restriction} $(A, \cl^A)$, where for $B\subseteq A$, $\cl^A(B)=\cl(B)\cap A$; and
    \item the \textbf{relativisation} $(\Omega, \cl_A)$, where for $B\subseteq \Omega$, $\cl_A(B)=\cl(B\cup A)$.
\end{itemize}
For $A\subseteq \Omega$, we write $\dim(A)$ for the dimension of $(A, \cl^A)$ and $\dim(\Omega/A)$ for the dimension of $(\Omega, \cl_A)$. We call the latter the \textbf{codimension} of $A$ in $\Omega$. 
{We extend this notation as follows: For $A,  B\subseteq\Omega$, setting $D:=\cl(A\cup B)$, we write  $\dim(B/A)$ for the codimension of $A$ in $(D,\cl^D)$.}
\end{definition}

\begin{definition} Let $\Omega$ be a set. Given a 
closure operator 
$\cl\colon \mathcal{P}(\Omega)\to\mathcal{P}(\Omega)$, we can define the ternary relation $\ind$ of \textbf{independence}, where for $A, B, C\subseteq \Omega$ we write 
\[A\ind_C B\quad \text{ if and only if }\quad \cl(A\cup C)\cap\cl(B\cup C)=\cl(C)\;.\]
    In the circumstances above, we say that $A$ is \textbf{independent} from $B$ over $C$. To simplify notation, if in the above setting $C=\emptyset$, we just write $A\ind  B$ {and say that $A$ is independent from $B$.} Slightly abusing notation, we also apply the notation $\ind$ to tuples and elements of $\Omega$.
    
\end{definition}

\begin{fact}\label{fact:indpregeom} Let $(\Omega, \cl)$ be a pregeometry. Let $C\subseteq \Omega$ be closed and let $D$ be the closure in $(\Omega, \cl)$ of a basis for $(\Omega, \cl_C)$. Then, $D\ind C$. 
\end{fact}
\begin{proof} 
{Fix a basis $B$ for $(\Omega, \cl_C)$. Suppose for contradiction that there exists $x\in (D\cap C)\setminus \cl(\emptyset)$. There is a minimal $B'\subseteq B$ such that $x\in\cl(B')$. We must have $B'\neq \emptyset$; for any $b\in B'$ we then have that $b\in\cl(B\setminus\{b\}\cup\{x\})$. Hence,  
$D=\cl(B\setminus\{b\}\cup\{x\})$. Since $x\in C$, this implies that $B\setminus\{b\}$ is a basis for $(\Omega, \cl_C)$, a contradiction.}
\end{proof}

\begin{definition}
We say that a pregeometry $(\Omega, \cl)$ is \textbf{modular} if for all closed $A, B\subseteq \Omega$,
\[\dim(A\cup B)+\dim(A\cap B)=\dim(A)+\dim(B)\;.\]
The operation above is cardinal addition since the dimension of a set can be infinite.
\end{definition}

\begin{fact}[{{Proposition 1.2.17} in~\cite{d2023axiomatic}}]\label{fact:equivalentsmodular} The following conditions are equivalent {for a pregeometry $(\Omega,\cl)$}:
\begin{itemize}
    \item $\cl$ is modular;
    \item for any closed $A, B\subseteq\Omega$ we have  $A\ind_{A\cap B} B$;
    \item whenever $A, B\subseteq\Omega$ and  $c\in\cl(A\cup B)$, then there are $a\in \cl(A)$ and $b\in \cl(B)$ such that $c\in\cl(\{a,b\})$.
\end{itemize}    
\end{fact}

\begin{examples} The following are examples of modular pregeometries:
\begin{itemize}
    \item any pregeometry {$(\Omega,\cl)$} where $\cl$ is \textbf{degenerate}:\\ for each $A\subseteq\Omega$ we have  $\cl(A)=\bigcup_{a\in A} \cl(a)$; 
    \item the pregeometry $(V, \cl)$ where $V$ is a vector space and $\cl$ is linear span;
    \item the pregeometry $(PV, \cl)$ where $PV$ is a projective space and $\cl$ is projective closure.
\end{itemize}
An example of a non-modular pregeometry is {$(\mathbb{C}, \cl)$ where $\mathbb{C}$ is the field of complex numbers and $\cl$ is algebraic closure~\cite[Example 1.8]{pillay1996geometric}.} 
\end{examples}

\begin{fact}\label{fact:modularity} Let $(\Omega, \cl)$ be a modular pregeometry. Let $A, B\subseteq\Omega$ be closed such that $A\ind B$ and suppose that
\[a\in\cl(A\cup B)\setminus(A\cup B)\;.\]
Then, $A\cup\{a\}\nind B$.
\end{fact}
\begin{proof} By Fact~\ref{fact:equivalentsmodular}, there are $a'\in A$ and $b'\in B$ such that $a\in\cl(\{a',b'\})$; neither $a'$ nor $b'$ can belong to $\cl(\emptyset)$. In particular, since $a\not\in A\cup B$, we have $a\in\cl(\{a',b'\})\setminus \cl(b')$. Thus, by exchange, $b'\in\cl(\{a,a'\})$. 
Hence, $\cl(A\cup \{a\})\cap B\supsetneq \cl(\emptyset)$, and so  $A\cup\{a\}\nind B$.
\end{proof}

\begin{fact}\label{fact:codimension} Let $(\Omega, \cl)$ be a modular pregeometry. Then, the intersection of finitely many closed sets of finite codimension has finite codimension.  
\end{fact}
\begin{proof}It is sufficient to show the fact for the intersection of two closed sets $A,B$ of finite codimension. We know that 
 $\dim(\Omega/A\cap B)=\dim(A/A\cap B)+\dim(\Omega/A)$~\cite[Lemma C.1.8]{tent2012course} (note these may be infinite cardinals). Since $A$ has finite codimension, $\dim(\Omega/A)$ is finite. By modularity,  Fact~\ref{fact:equivalentsmodular} implies that $A$ and $B$ are independent over $A\cap B$; this easily implies   $\dim(A/B)=\dim(A/A\cap B)$~\cite[Lemma C.1.10]{tent2012course}. Since $B$ has finite codimension in $\Omega$, $\dim(A/B)$ is finite. Hence, $\dim(\Omega/A\cap B)$ is finite, as desired.
\end{proof}




{The following fact is folklore:}

\begin{fact}\label{fact:axioms} 
{Let $\Omega$ be a set, and let $\cl$ be a closure operator on $\Omega$. Then, $\ind$ satisfies the following properties for all $A,B,C,D\subseteq \Omega$:}

\begin{tabular}{ll}
  {\sc(existence)} &   $A\indep{C}{} C$;\\
  {\sc(normality)} & if $A\indep{C}{} B$, then $A\indep{C}{} B\cup C$;\\
  {\sc(symmetry)}&   $A\indep{C}{}B$ if and only if $B\indep{C}{}A$;\\
   {\sc(monotonicity)} &  if $A\indep{C}{} B\cup D$ then $A\indep{C}{} B$;\\
   {\sc(transitivity)} & {if $B\subseteq C\subseteq D$, $A\indep{B}{} C$,  and $A\indep{C}{} D$, then $A\indep{B}{} D$.}
\end{tabular}

Suppose further that $(\Omega, \cl)$ forms a pregeometry. Then, $(\Omega, \cl)$ is modular if and only if $\ind$ further satisfies the following for all $A,B,C,D\subseteq \Omega$: 

\begin{tabular}{ll}
{\sc(base monotonicity)} & if {$B\subseteq C\subseteq D$ and   $A\indep{B}{} D$,  then $A\indep{C}{} D$.}
\end{tabular}

\end{fact}
\begin{proof}
   {Existence, normality, symmetry, and monotonicity follow straightforwardly from the definition. Transitivity is also a routine check: suppose that $d\in\cl(A\cup B)\cap\cl(D)$. Firstly, since $A\ind_C D$, we must have that $d\in\cl(C)$. And since $A\ind_B C$,  $d\in\cl(B)$. Hence, $\cl(A\cup B)\cap \cl(D)=\cl(B)$, and so $A\ind_B D$ as desired. The final statement 
    on base monotonicity is standard; see~\cite[Proposition 1.2.12]{d2023axiomatic} for a proof.}
\end{proof}
Below, we prove a basic fact about 
{independence} in a modular context from the statements given in Fact~\ref{fact:axioms}. This will be helpful in our computations in the proof of Lemma~\ref{lem:hardinduction}. 
\begin{lemma}\label{lem:modularmoving} Suppose that {$(\Omega, \cl)$} is a modular pregeometry. {Then, for all $U,V,C\subseteq \Omega$ such that  $U\ind V$ and $C\ind U$}, we have that
 \[C\indep{}{} U\cup V\quad \text{ if and only if }\quad  C\cup U\indep{}{} V\;.\]
\end{lemma}
\begin{proof}  We have the following equivalences, which we explain below:
\begin{align}
        C\indep{}{} U\cup V\quad & \Leftrightarrow\quad  C\indep{U}{} U\cup V \\
        &\Leftrightarrow\quad  C\indep{U}{} V \\ 
         & \Leftrightarrow\quad  C\cup U\indep{}{} V 
    \end{align}   
   For the first equivalence, $\Rightarrow$ follows by base monotonicity and $\Leftarrow$ by transitivity and monotonicity. For the second equivalence, $\Rightarrow$ follows by monotonicity and $\Leftarrow$ by normality. For the final equivalence, $\Rightarrow$ follows by transitivity, and $\Leftarrow $ by {base monotonicity}.
\end{proof}



{The following notion, which notably makes sense also when $\cl$ does not satisfy exchange, is  a generalization of Neumann's lemma~\cite{neumann1976structure}.}
The latter states that if $G\acts \Omega$, and $P, C\subseteq \Omega$ are finite, and all elements of $P$ have an infinite orbit, then there is $g\in G$ such that $g\cdot  P$ avoids $C$. In our generalization which we call ``higher Neumann'', the set $g\cdot P$ is further spread out by shifts 
{from} a finite set $\Sigma\subseteq G$. In the formal definition below, we shall content ourselves with a one-element $P=\{a\}$ as this is sufficient for our purposes. We shall give sufficient conditions  implying  higher Neumann in Section~\ref{sect:superspacious}, but mention already here  that the property is satisfied by $\acl$ in the situation of a group action with no algebraicity and  by linear span for 
general linear groups in their natural action on  vector spaces; {moreover, it is satisfied by projective closure for projective {linear} groups in their action on projective spaces.}} 

\begin{definition}\label{def:higherneumann}
   {Let $G\acts\Omega$ be a group action, and  let $\cl$ be a $G$-invariant closure operator on $\Omega$.} We say that 
   $(\Omega, {\cl})$ is \textbf{higher Neumann}
    if for each $B\in[\Omega]^{<\omega}$, $a\in\Omega$ such that $a\ind B$, $C\in[\Omega]^{<\omega}$, and $\Sigma\subseteq G$ finite, there is some $a'\equiv_B a$ such that 
\[\Sigma \cdot  a'\ind C\;.\]
\end{definition}

\begin{remark}\label{rem:minimalityOfacl}
    {If $\cl$ is a $G$-invariant closure operator such that $(\Omega, \cl)$ forms a higher Neumann pregeometry, then for all $B\subseteq \Omega$, $\acl(B)\subseteq\cl(B)$. We only have to verify this for finite $B$, as both operators are algebraic  by definition. In the following, we use $\ind$ with respect to $\cl$ (not $\acl$). Let $a\in\acl(B)$, i.e.~$G_B\cdot a$ is finite. {If $a\in\cl(\emptyset)$, then $a\in\cl(B)$, so we may assume this is not the case.} We then  cannot have that $a\ind B$, since 
     {clearly there} is no $a'\equiv_B a$ such that $a'\ind (B\cup G_B\cdot a)$, {and} $(\Omega,\cl)$ {is higher Neumann}. Hence, $a\nind B$, i.e.~there exists  $c\in\cl(a)\cap\cl(B)\setminus\cl(\emptyset)$. But, by exchange, $a\in\cl(c)$, and so $a\in\cl(B)$. Thus, $\acl(B)\subseteq\cl(B)$.}
\end{remark}

    


    

%% file: Sections/mixed.tex
\section{Proof of the Main Theorem}\label{sect:proof}

\subsection{Rough strategy and notation setup}
Throughout this section we work with a 
 {word} $w(x_1,\ldots,x_r)$ as in (\ref{eq:word}). 

Given such word, assuming regularity we will want to find group elements 
$g_1,\dots,  g_r\in G$ such that $w(g_1, \dots, g_r)\neq 1$. 
{Using the action $G\acts\Omega$,} 
this amounts to identifying  some $a_0\in\Omega$ such that $w(g_1, \dots, g_r)(a_0)\neq a_0$. In fact, {in order to obtain density 
of the non-solutions to $w(x_1,\ldots,x_r)=1$}, we will even want to show that $(g_1, \dots, g_r)$ may be found in an arbitrary non-empty basic open subset $G_{(\bar{b}_1, \bar{b_1')}}\times\dots\times  G_{(\bar{b}_r, \bar{b}_r')}$ of $G^r$. 

We will construct $g_1,\ldots,g_r$ inductively by considering longer and longer initial segments of $w$ (from the right). In the beginning, we know that each of the $g_i$ should belong to a given basic open set, which means that it is  determined on a finite set of elements; we call these constraints on the $g_i$ the   ``initial condition''. Suppose we have picked $g_1,\ldots,g_r$ satisfying the initial condition. We then pick an element $a_0$ and follow the values of $u(g_1, \dots, g_r)\cdot a_0$ as we apply increasingly longer initial segments $u$ of $w$; 
this gives rise to the notion of a ``sequence''. Clearly, to compute the sequence, one only needs finite information about the $g_i$. The central idea of our induction is to start with the initial condition, pick the value $a_0$ wisely, and to then increase the conditions on the $g_i$ step by step in such a way that the sequence stays as free as possible, allowing us to terminate  at a value different from $a_0$  when $w$ is regular.  
Figure~\ref{fig:sequence} shows a useful way to depict the trajectory of an element $a_0$ through $w(g_1,\ldots,g_r)$, and the intuition behind a sequence.

\begin{definition}\label{def:sequence} An  \textbf{initial condition} $B=((\bar{b}_1, \bar{b}_1'),\ldots,(\bar{b}_r, \bar{b}_r'))$ is a finite sequence consisting of pairs $(\bar{b}_i, \bar{b}_i')$  of finite {(possibly empty)}  tuples from $\Omega$ with $\bar{b}_i\equiv \bar{b}_i'$. 
 Let $a_0\in\Omega$. Let $w\in G* F_r$ be a 
 word as in (\ref{eq:word}) and $u\leq w$ be a {possibly empty} initial segment (from the right)  of $w$ of the form 
\begin{equation}\label{eq:subword}
   \gamma_\ell\ x_{\iota_\ell}^{\epsilon_\ell}\ \gamma_{\ell-1}\  x_{\iota_{\ell-1}}^{\epsilon_{\ell-1}}\ \dots\ \gamma_1\ x_{\iota_{1}}^{\epsilon_1}\ \gamma_0\;, 
\end{equation}
for $\ell\leq n$. 
A \textbf{sequence} for $u$ with initial condition $B$ and initial value $a_0$ 
is a sequence $\overline{a}:=(a_0, a_1, a_1'\dots, a_\ell, a_\ell', a^{\epsilon_{\ell+1}}_{\ell+1})$ of elements of $\Omega$  
satisfying the following conditions:
    \begin{enumerate}
    \item For each $i\in\{1, \dots, r\}$,  
    letting $j_1,\ldots,j_m$ enumerate all  those elements in 
    {$\{1, \dots, \ell\}$} where $\iota_{j_1}=\cdots=\iota_{j_m}=i$, 
    \[(\bar{b}_i,a_{j_1},\dots, a_{j_m})\equiv (\bar{b}_i',a'_{j_1},\dots, a'_{j_m})\;;\]

\item\label{it:badconv} setting the conventions that {$a_1^{\epsilon_1}=\gamma_0 a_0$ and } 
{$\epsilon_{n+1}=1$}, and     writing $a_j^{-1}$ for $a_j'$,  and $a_j^1$ for $a_j$, 
\[a_{j+1}^{\epsilon_{j+1}}=\gamma_j {a_j^{-\epsilon_{j}}}\;\]
for all {$j$ such that $1\leq j\leq \ell$}.
\end{enumerate}  
We call $a_{\ell+1}^{\epsilon_{\ell+1}}$ the \textbf{terminal value} of the sequence $\bar{a}$. 
\end{definition}

Intuitively, the first condition states that each $a_j'$ can be obtained from $a_j$ by applying an element of $G$; that these group  elements can be chosen to be the same for all positions $j$ where the same variable appears in $u$; and that that {uniform witness} can be chosen to satisfy the initial condition. The second item explains the relationship between $a_j$ and $a_{j+1}$, which are separated by $\gamma_j$ as we follow the trajectory of $a_0$ through $u$.

The following is a trivial consequence of Definition~\ref{def:sequence}: 
\begin{lemma}\label{lem:sequences} Let $B$ be {an} initial condition and let $a_0, a_{n+1}\in\Omega$. Let $w\in G*F_r$ be a 
{word} as in (\ref{eq:word}). Then, there are $(g_1, \dots, g_r)\in {\prod_{1\leq i\leq r} G_{(\bar{b}_i, \bar{b}_i')}}$ such that $w(g_1, \dots, g_r)(a_0)=a_{n+1}$ if and only if there is a sequence $\bar{a}$ for $w$ with initial condition $B$, initial value $a_0$, and terminal value $a_{n+1}$.
\end{lemma}

\begin{notation} We naturally extend the notion of freeness rank 
to the elements of a sequence $\overline{a}$ for an initial segment $u$ of $w$ with initial condition $B$. Namely, if we imagine the sequence to arise by following the trajectory of $a_0$ through the word $u(g_1,\ldots,g_r)$ by applying increasingly long initial segments of $u(x_1,\ldots,x_r)$ (evaluated at $(g_1,\ldots,g_r)$) to $a_0$, then the freeness rank of an element of the sequence is the freeness rank of the initial segment at which the element appears in this trajectory. Note that since we do not require sequences to be injective, this initial segment  is actually not necessarily unique. However, the sequences we build will be ``rank-clean'' (see Definition~\ref{def:rankclean}), which implies that the freeness rank of 
all initial segments yielding the same value have the same freeness rank. 
Formally, let  $w$ be a word as in (\ref{eq:word}), let $u$ be an initial segment of $w$ as in~(\ref{eq:subword}), and let $(a_0, a_1, a_1'\dots, a_\ell, a_\ell', a^{\epsilon_{\ell+1}}_{\ell+1})$ be a sequence for $u$. Then we set  
\begin{align}
\rk(a_k^{\epsilon_k}) &= \rk(\gamma_{k-1}\ x^{\epsilon_{k-1}}_{\iota_{k-1}}\ \dots\  \gamma_1\ x_{\iota_1}^{\epsilon_1}\ \gamma_0)  &\text{ for all }  
{k\in\{1, \dots, \ell+1\}}
\text{ and }\\
\rk(a_k^{-\epsilon_k}) &=\rk(\gamma_{k}\ x^{\epsilon_k}_{\iota_k}\ \dots\  \gamma_1\ x_{\iota_1}^{\epsilon_1}\gamma_0) &\text{ for all }  {k\in\{1, \dots, \ell\}\;.}
\end{align}
setting by convention that $\rk(a_0)=1_{F_r}$. 
\end{notation}

\input{Pictures/figuresequence}

\begin{observation} Let $\overline{a}$ be a sequence for $w$ with initial condition $B$. A feature of our notation is that  {$\iota_j=\iota_k$ and} $\rk(a_j^\theta)=\rk(a_k^\theta)$ implies  $\rk(a_j^{-\theta})=\rk(a_k^{-\theta})$ for all {$1\leq j<k\leq n$ }
and all $\theta\in\{-1, 1\}$.  
\end{observation}

%% file: Pictures/figuresequence.tex
\begin{figure}
    \centering
    \begin{tikzpicture}[scale=1.3]

\node[anchor=south, inner sep=5pt, gray] at (5.5, 6.5){Rank $0$};
\node[anchor=south, inner sep=5pt, gray] at (3.5, 6.5){Rank $1$};
\node[anchor=south, inner sep=5pt, gray] at (1.5, 6.5){Rank $2$};
\node[anchor=south, inner sep=5pt, gray] at (-0.5, 6.5){Rank $3$};
\node[anchor=south, inner sep=5pt, gray] at (-2.5, 6.5){Rank $4$};

\draw[thick, dashed, gray] (4.5, 1)--(4.5, 7);
\draw[thick, dashed, gray] (2.5, 1)--(2.5, 7);
\draw[thick, dashed, gray] (0.5, 1)--(0.5, 7);
\draw[thick, dashed, gray] (-1.5, 1)--(-1.5, 7);

\draw[->, thick, Oxred] (6,6)--(5,6);
\draw[->, thick, Oxblue] (5,6)--(4,6);
\draw[->, thick, Oxred] (4,6)--(3,6);
\draw[->, thick, Oxblue] (3,6)--(2,6);
\draw[->, thick, Oxred] (2,6)--(1,6);
\draw[->, thick, Oxblue] (1,6)--(0,6);
\draw[->, thick, Oxred] (0,6)--(0,5);
\draw[->, thick, Oxblue] (1,5)--(0,5);
\draw[->, thick, Oxred] (1,5)--(2,5);
\draw[->, thick, Oxblue] (3,5)--(2,5);
\draw[->, thick, Oxred] (3,5)--(3,4);
\draw[->, thick, Oxblue] (3,4)--(2,4);
\draw[->, thick, Oxred] (2,4)--(1,4);
\draw[->, thick, Oxblue] (1,4)--(0,4);
\draw[->, thick, Oxred] (0,4)--(-1,4);
\draw[->, thick, Oxblue] (-1,4)--(-2,4);
\draw[->, thick, Oxred] (-2,4)--(-2,3);
\draw[->, thick, Oxblue] (-1,3)--(-2,3);
\draw[->, thick, Oxred] (-1,3)--(0,3);
\draw[->, thick, Oxblue] (1,3)--(0,3);
\draw[->, thick, Oxred] (1,3)--(1,2);
\draw[->, thick, Oxblue] (1,2)--(0,2);
\draw[->, thick, Oxred] (0,2)--(-0.5, 1.5);
\filldraw (6, 6) circle (1pt) node[anchor=south, inner sep=5pt] {\normalsize$a_0$};
\node[anchor=south, inner sep=5pt, Oxred] at (5.5, 6){\normalsize$\gamma_0$};
\filldraw (5, 6) circle (1pt) node[anchor=south, inner sep=5pt] {\normalsize$a_1$};

\filldraw (4, 6) circle (1pt) node[anchor=south, inner sep=5pt] {\normalsize$a_1'$};
\node[anchor=south, inner sep=5pt, Oxred] at (3.5, 6){\normalsize$\gamma_1$};
\filldraw (3, 6) circle (1pt) node[anchor=south, inner sep=5pt] {\normalsize$a_2$};

\filldraw (2, 6) circle (1pt) node[anchor=south, inner sep=5pt] {\normalsize$a_2'$};
\node[anchor=south, inner sep=5pt, Oxred] at (1.5, 6){\normalsize$\gamma_2$};
\filldraw (1, 6) circle (1pt) node[anchor=south, inner sep=5pt] {\normalsize$a_3$};

\filldraw (0, 6) circle (1pt) node[anchor=east, inner sep=5pt] {\normalsize$a_3'$};
\node[anchor=east, inner sep=5pt, Oxred] at (0, 5.5){\normalsize$\gamma_3$};
\filldraw (0, 5) circle (1pt) node[anchor=east, inner sep=5pt] {\normalsize$a_4'$};
\filldraw (1, 5) circle (1pt) node[anchor=south, inner sep=5pt] {\normalsize$a_4$};
\node[anchor=south, inner sep=5pt, Oxred] at (1.5, 5){\normalsize$\gamma_4$};
\filldraw (2, 5) circle (1pt) node[anchor=south, inner sep=5pt] {\normalsize$a_5'$};

\filldraw (3, 5) circle (1pt) node[anchor=west, inner sep=5pt] {\normalsize$a_5$};
\node[anchor=west, inner sep=5pt, Oxred] at (3, 4.5){\normalsize$\gamma_5$};
\filldraw (3, 4) circle (1pt) 
node[anchor=west, inner sep=5pt] {\normalsize$a_6$};
\filldraw (2, 4) circle (1pt) 
node[anchor=south, inner sep=5pt] {\normalsize$a_6'$};
\node[anchor=south, inner sep=5pt, Oxred] at (1.5, 4){\normalsize$\gamma_6$};
\filldraw (1, 4) circle (1pt) 
node[anchor=south, inner sep=5pt] {\normalsize$a_7$};
\filldraw (0, 4) circle (1pt) 
node[anchor=south, inner sep=5pt] {\normalsize$a_7'$};
\node[anchor=south, inner sep=5pt, Oxred] at (-0.5, 4){\normalsize$\gamma_7$};
\filldraw (-1, 4) circle (1pt) 
node[anchor=south, inner sep=5pt] {\normalsize$a_8$};
\filldraw (-2, 4) circle (1pt) 
node[anchor=east, inner sep=5pt] {\normalsize$a_8'$};
\node[anchor=east, inner sep=5pt, Oxred] at (-2, 3.5){\normalsize$\gamma_8$};
\filldraw (-2, 3) circle (1pt) 
node[anchor=east, inner sep=5pt] {\normalsize$a_9'$};
\filldraw (-1, 3) circle (1pt) 
node[anchor=south, inner sep=5pt] {\normalsize$a_9$};
\node[anchor=south, inner sep=5pt, Oxred] at (-0.5, 3){\normalsize$\gamma_9$};
\filldraw (0, 3) circle (1pt) 
node[anchor=south, inner sep=5pt] {\normalsize$a_{10}'$};
\filldraw (1, 3) circle (1pt) 
node[anchor=west, inner sep=5pt] {\normalsize$a_{10}$};
\node[anchor=west, inner sep=5pt, Oxred] at (1, 2.5){\normalsize$\gamma_{10}$};
\filldraw (1, 2) circle (1pt) 
node[anchor=west, inner sep=5pt] {\normalsize$a_{11}$};
\filldraw (0, 2) circle (1pt) 
node[anchor=south, inner sep=5pt] {\normalsize$a_{11}'$};
\node[anchor=west, inner sep=5pt, Oxred] at (-0.3, 1.7){\normalsize$\gamma_{11}$};
\filldraw (-0.5, 1.5) circle (1pt) 
node[anchor=east, inner sep=5pt] {\normalsize$a_{12}$};

 
\end{tikzpicture}
   
    \caption{Visual representation of the sequence $(a_0, \dots, a_{12})$ for the one-variable word
    \begin{equation*}
w(x)=\gamma_{11}x\gamma_{10}x^{-1}\gamma_9x^{-1}\gamma_8x\gamma_7x\gamma_6x\gamma_5x^{-1}\gamma_4x^{-1}\gamma_3x\gamma_2x\gamma_1x\gamma_0   
    \end{equation*}
    evaluated at $g\in G$. {The labelled red arrows indicate the action of constants from the word.} The unlabelled blue arrows, always going right-to-left, indicate the action of $g$ on elements of the sequence. When the arrows cross the gray dotted line, this indicates a change in the freeness rank of the corresponding initial segment of the word $w$ {(equivalently, of the corresponding element of the sequence).}
    }
    \label{fig:sequence}
\end{figure}

%% file: Sections/Maintheorem.tex
\subsection{The strategy in more detail, and the actual proof}


{Throughout this section we work with a group action  $G\acts\Omega$ and a 
 closure operator $\cl$ on $\Omega$ such that $(\Omega, \cl)$ forms a $G$-invariant higher Neumann infinite-dimensional modular pregeometry.}

The core idea of our proof is that for a  
{word}  $w$ and initial condition $B$, we inductively build a sequence $\overline{a}$ for $w$ whose  elements of different freeness rank are sufficiently independent from each other. If $w$ is regular, this then yields $a_0\neq a_{n+1}$ since their freeness ranks differ. Hence,  $w$ does not yield a mixed identity, and this is witnessed by group elements belonging to the basic open sets as given by the initial condition. 

Naively, one would be tempted to establish said independence condition  at every step, without looking forward in the future trajectory of the sequence: just pick the next element of the sequence to be  independent from all previous elements of different freeness rank. 
The main issue that arises is that sometimes later elements in our sequence may be entirely determined by previous elements. This is  due to the presence of critical constants.  To see this, consider the context of Figure~\ref{fig:sequence}, and suppose that we have already built the sequence $(a_0, \dots, a_3, a_3', a_4')$. It may be that $\gamma_3$ fixes $a_3'$ and so $a_4=a_3$ and $a_5'=\gamma_4\gamma_2 a_2'$. By our desired independence condition, we want $a_5'$ to still be independent from $a_0$. So, for this to be the case, when we constructed $a_2'$, we must have already ensured that $\gamma_4\gamma_2 a_2'$ is independent from $a_0$. This in itself can already be ensured by Neumann's lemma. When more than one critical constant occurs, we need to deal with additional issues which require the higher Neumann property. For example, consider the case where we have built the sequence $(a_0, \dots, a_7)$, again in the context of Figure~\ref{fig:sequence}. By construction, $(a_1, \dots, a_5)\equiv (a_1',\dots, a_5')$. We could have that $\gamma_5 a_5=a_6\in {\cl}(\{a_2,  a_5\})$ (we cannot ensure independence amongst elements of the same freeness rank since  some singular words might be mixed identities). But then, this would mean that $a_7'\in {\cl}(\{\gamma_6 a_5',\gamma_6 a_2\})$. So in order to ensure that $a_7'\neq a_0$, we need to also ensure that the 
{closure} of translates of elements of a given rank is still independent from elements of other ranks. The definition of higher Neumann is designed to ensure the satisfaction of independence conditions of this kind. In order to formally state the sophisticated independence conditions we need in our induction (which we shall call having a rank-clean sequence), below we define the notions of paths, foreseeable future, and foreseeable closures.

\begin{definition}\label{def:rankclean}
Let $w\in G*F_r$ be a 
{word}  of the form (\ref{eq:word}),  and let {$j\leq n+1$}. Then, the set of \textbf{paths} at stage $j$ is given by
\[\mathfrak{P}(w, j):=\{1\}\cup\{\gamma_{i_0}^{\delta_0}\cdots \gamma_{i_k}^{\delta_k} \ \vert \ 
i_0, \dots, i_k\geq j\;
\text{ are distinct and }\; \delta_0, \dots, \delta_k\in\{-1,1\} 
\}\;.\]
Note that these sets are finite and decrease as $j$ increases, {with $\mathfrak{P}(w, n+1)=\{1\}$}. 
Let ${\overline{a}}:=(a_0,a_1,a_1',\ldots,a_{\ell-1},a_{\ell-1}',a_\ell)$ be a sequence for an initial segment of  $w$, and let $a_i^\epsilon$ be an element of this sequence. The \textbf{foreseeable future} of $a_j^\epsilon$ in the word $w$  at stage $j$, denoted by $F(a_j^\epsilon, j)$, 
is the set of all shifts of $a_j^\epsilon$ by paths at stage $j$ (we add the index  $j$ as a parameter since the same element could appear multiple times in the sequence, and this cannot be avoided): 
\[F(a_j^\epsilon ,j):=\mathfrak{P}(w, j) \cdot  a_j^\epsilon\;.\]

{Let $X$ be a set of elements from the sequence $\overline{a}$. The \textbf{foreseeable closure} of $X$ is the 
 {closure} of the foreseeable futures of elements in $X$.}
We denote this by 
\[FC(X):={\cl}\left(\bigcup_{a_j^\epsilon \in X} F(a_j^\epsilon, j)\right).\]
{(Again, with a slight abuse of notation we record the position of the elements of $X$ in order to allow for repeated elements in the sequence to be counted as distinct in $X$.)}

The sequence $\overline{a}$ for the initial segment $u$ of $w$ and with initial condition $B$ is \textbf{rank-clean}  {if} 
\begin{itemize}
    \item no element of $\overline{a}$ is contained in ${\cl}(\emptyset)$,
    \item 
    {the foreseeable closure of the set of all elements of $\overline{a}$}  
     is independent from $B$, and
    \item for any partition $(P,Q)$ of the (finite) set of all freeness ranks of elements of $\overline{a}$, the foreseeable closure of elements with rank in $P$ is independent over $B$ from the foreseeable closure of elements with rank in $Q$.
\end{itemize}
Here, we are {being slightly informal} 
 since an initial condition is a sequence of pairs of tuples and not a set, but in this {context} we refer to the set of all elements that appear in it. Note that {the first condition is equivalent to $a_0$ not belonging to  ${\cl}(\emptyset)$ since
${\cl}$ is G-invariant.}
\end{definition}

\begin{lemma}\label{lem:basecase} {Let $G\acts \Omega$ be a group action and let  $(\Omega, {\cl})$ be a {$G$-invariant} higher Neumann  infinite-dimensional  pregeometry.} Let $w$ be a 
{word} and $B$ be {an} initial condition. Then, there  is a one-element rank-clean sequence $(a_0)$ with initial condition $B$ for the initial segment $\emptyset$ of $w$.
\end{lemma}
\begin{proof} 
{Since $(\Omega,{\cl})$ is infinite-dimensional, there exists $a\in\Omega {\setminus {\cl}(\emptyset)}$ such that $a\ind B$ (here we apply the above-mentioned abuse of notation and really refer to the set of  elements appearing in $B$). Since $(\Omega,{\cl})$ is higher Neumann,} we can pick $a_0\equiv_B a$ such that
\[\mathfrak{P}(w, 0)\cdot  a_0\;\ind\;  B\;.\]
In particular, 
\[FC(\{a_0\})\ind B\;,\]
and this is the only independence condition {in addition to $a_0\notin {\cl}(\emptyset)$} we are requiring of a rank-clean sequence with initial condition $B$ for $\emptyset\leq w$ (since at this stage the only rank witnessed by our sequence consists of $1_{F_r}$). 
\end{proof}

\begin{lemma}\label{lem:hardinduction} 
{Let $G\acts \Omega$ be a group action and let $(\Omega, {\cl})$ be a $G$-invariant  higher Neumann infinite-dimensional modular pregeometry.} 
Let $w$ be a 
word and let $B$ be {an} initial condition. 
 Let $(a_0)$ be a one-element rank-clean sequence over $B$ for the initial segment $\emptyset$ of $w$. 
 Then, there is a rank-clean  sequence $\overline{a}$ for $w$ with initial condition $B$ and initial value $a_0$.  
\end{lemma}
\begin{proof}
    We prove by induction over the length of an initial segment $u$ of $w$ that we can find a rank-clean sequence for $u$ with initial condition $B$ and initial value $a_0$. Clearly, for $u=w$ this then yields the lemma. Since the initial condition $B$ and the initial value $a_0$ are  fixed throughout the proof, we shall just talk about  sequences for initial segments of $w$, leaving it implicit that these have initial condition $B$ and initial value $a_0$.

    The base case of 
    $u=\gamma_0$ 
    is taken care of in the assumption that  the one-element sequence $(a_0)$ for the empty initial segment is rank-clean for the initial condition $B$. {To see this, note that since  $a_1^{\epsilon_1}=\gamma_0 a_0$, we have  $\rk(a_0)=\rk(a_1^{\epsilon_1})$ and $\mathrm{FC}(\{a_0\})=\mathrm{FC}(\{a_0, a_1^{\epsilon_1}\})$.}

    In the induction step, suppose that  $u=\gamma_\ell\  x_{\iota_\ell}^{\epsilon_\ell}\ v$, where the induction hypothesis holds for $v$. By renaming variables, we may assume that  $\iota_\ell=1$. 
     We discuss the case $\epsilon_\ell=1$ and leave  the case $\epsilon_\ell=-1$ to the reader.\footnote{The proof of the case with $\epsilon_\ell=-1$ is identical to the one we proceed to give except for some minor modifications. We briefly sketch these here so that the reader may apply these when verifying this other case: when defining a rank-clean sequence for $v$ like we do in the line of the proof following this footnote, its terminal element is $a'_\ell$ instead of $a_\ell$ due to our choice of notation. Hence, in Case 1, when extending the sequence to a new element in (\ref{eq:extend1}), this is $a_\ell=\delta^{-1} a_\ell'$ on the left hand side of the equation. Similarly, when extending the sequence in Case 2, instead of finding $a''_\ell=\delta a_\ell$ where $\delta$ is a witness of (\ref{eq:extend2}), we are finding $a^\star_\ell=\delta^{-1}a_\ell'$, where $\delta$ is a witness of (\ref{eq:extend2}). The rest of the computations are the same by switching the superscript $-1$ and $1$ and $''$ by $^\star$.} Let $(a_0,a_1,a_1', \dots,a_{\ell-1}, a_{\ell-1}', a_\ell)$ be a  rank-clean sequence for $v$ with initial condition $B$. We distinguish two cases:
    \begin{itemize}
    \item Case 1: The final element $a_\ell$ is contained in the 
    {closure} of previous elements to which the same variable $x_1$ is applied in $w$ (together with $B$); that is,  
    \[a_\ell\in {\cl}\left(B\cup \{a_j\;|\; j<\ell,\; \iota_j=1\}\right)\; ; \]
    \item Case 2: this is not the case.
\end{itemize}

We begin by considering \textbf{Case~1}.

\textbf{Claim~1.1:} $a_\ell$ is contained in the 
{closure} of the set 
$$X:=\{a_j\;|\; j<\ell,\; \iota_j=1,\; \rk(a_j)=\rk(a_\ell)\}\; .$$ 
\begin{subproof}[Proof of Claim~1.1] Write  $$Y:=\{a_j\;|\; j<\ell,\; \iota_j=1,\; \rk(a_j)\neq \rk(a_\ell)\}\; .$$ 
By inductive hypothesis, $\{a_\ell\}\cup X\;\ind_B\; Y$ and $\{a_\ell\}\cup X\;\ind\; B$. Thus, by transitivity in the sense of Fact~\ref{fact:axioms}, 
\begin{equation}\label{eq:claim1indep}
  \{a_\ell\}\cup X\; \ind\; B\cup Y \;. 
\end{equation} 
Consider  Fact~\ref{fact:modularity}, substituting $A$ by ${\cl}(X)$, $B$ by ${\cl}(B\cup Y)$, and $a$ by $a_\ell$ in its statement. We then know that the conclusion of that  fact does not hold, by~(\ref{eq:claim1indep}). We also know that $a_\ell\in {\cl}({\cl}(X)\cup {\cl}(B\cup Y))$ by the assumption for Case~1. Hence, the fact implies $a_\ell\in {\cl}(X)\cup {\cl}(B\cup Y)$. Again by~(\ref{eq:claim1indep}), the only way for this to happen is $a_\ell\in {\cl}(X)$.


\end{subproof}


Set $a_\ell':=\delta a_\ell$ for an arbitrary witness $\delta$ of 
\begin{equation}\label{eq:extend1}
    (\bar{b}_1,a_{m_1},\dots, a_{m_s}) 
    \equiv (\bar{b}_1',a'_{m_1},\dots, a'_{m_s}) 
    \;,
\end{equation}
where $(m_1, \dots, m_s)$ is the tuple 
of indices denoting the elements of the sequence prior to $a_\ell$ to which the variable $x_1$ was applied (i.e.~they enumerate the indices {$1\leq j < \ell$} 
for which $\iota_j=1$). Setting 
$$X':=\{a'_j\;|\; j<\ell,\; \iota_j=1,\; \rk(a'_j)=\rk(a'_\ell)\}\; $$ 
we have $a_\ell'\in {\cl}(X')$ since  $a_\ell\in {\cl}(X)$ by Claim~1.1, 
$X'=\delta(X)$, and by the $G$-invariance of ${\cl}$. 
 Now set $a_{\ell+1}:=\gamma_\ell a_\ell'$. Then $(a_0,a_1,a_1',\ldots,a_\ell,a_\ell',a_{\ell+1})$ is a sequence for $u$ with initial condition $B$ by construction. We want to prove that it is a rank-clean sequence for $u$; this will be Claim~1.3 below.\\ 
 

\textbf{Claim~1.2:} Let $P$ be any set of freeness ranks which contains  $\rk(a_{\ell+1})$. For the set\\ $S:=\{a_0,a_1,a_1', \dots,a_{\ell-1}, a_{\ell-1}', a_\ell\}$ we have:
\begin{equation}\label{eq:foreseeable}
    FC(\{x\; \vert\;   x\in S\cup \{a_\ell', a_{\ell+1}\}, \rk(x)\in P\})=FC(\{x\;\vert\; x\in S, \rk(x)\in P\})\;.
\end{equation}
\begin{subproof}[Proof of Claim~1.2]
Write $U$ and $U^\star$ for the sets on the left and right hand sides of equation (\ref{eq:foreseeable}). The inclusion $U^\star\subseteq U$ is obvious, so we only need to prove that $U\subseteq U^\star$. 
We shall prove that $F(a_\ell', \ell)\subseteq U^\star$ and $F(a_{\ell+1}, \ell+1)\subseteq U^\star$. From this, the above equation, and transitivity of 
{closure} we obtain that $U\subseteq U^\star$. 

To see that $F(a_\ell', \ell)\subseteq U^\star$ note that
  \[F(a_\ell', \ell)\;=\;\mathfrak{P}(w, \ell)\cdot a_\ell'
    \;\subseteq\; \mathfrak{P}(w, \ell)\cdot {\cl}(X')\;{\subseteq}
    \; {\cl}(\mathfrak{P}(w, \ell)\cdot  X')\;\subseteq \;U^\star\;.\]
Here the first equality is by definition; the  first containment follows from the consequence of Claim~1.1 that $a_\ell'\in{\cl}(X')$;  the {next} {containment}
is just a consequence of the {$G$-invariance of $\cl$. The}
final containment follows from the fact that $$\rk(a_\ell')=x_1\cdot\rk(a_\ell)=\rk(a_{\ell+1})\; ,$$ implying that all elements in $X'$ have rank in $P$, and the definition of foreseeable closure. 

To see that $F(a_{\ell+1}, \ell+1)\subseteq U^\star$, note that 
\[F(a_{\ell+1}, \ell+1)\;=\;\mathfrak{P}(w, \ell+1)\cdot  a_{\ell+1}\;=\;\mathfrak{P}(w, \ell+1)\cdot \gamma_\ell a_\ell'\;\subseteq\; \mathfrak{P}(w, \ell) \cdot a'_\ell\;=\; F(a_\ell', \ell)\;\subseteq\; U^\star\;.\]
In this case, the first two equalities are by definition; the  containment thereafter follows from the fact that any product of the form $\alpha\gamma_\ell$ for $\alpha\in \mathfrak{P}(w, \ell+1)$ will be contained in $\mathfrak{P}(w, \ell)$; the next equality is by definition; and the final containment was just proven above. 

\end{subproof}

\textbf{Claim~1.3:} $(a_0,a_1,a_1',\ldots,a_\ell,a_\ell',a_{\ell+1})$ is  a rank-clean sequence for $u$.
\begin{subproof}[Proof of Claim~1.3]
Let $(P, Q)$ be any partition of the freeness ranks occurring in the sequence. We may assume without loss of generality that $P$ contains $\rk(a_\ell')=x_1\cdot \rk(a_\ell)=\rk(a_{\ell+1})$. 
Let $S$ be as in Claim~1.2. Since by inductive hypothesis the sequence  $(a_0,a_1,a_1', \dots,a_{\ell-1},a_{\ell-1}', a_\ell)$ is rank-clean,
\[FC(\{x\;\vert\; x \in S,\; \rk(x)\in P\})\ind_B FC(\{x\;\vert\; x \in S,\; \rk(x)\in Q\})\;.\]
However, since $\rk(a_\ell')=\rk(a_{\ell+1})\in P$, 
\[FC(\{x\;\vert\; x \in S,\; \rk(x)\in Q\})=FC(\{x\;\vert\; x \in S\cup\{a_\ell', a_{\ell+1}\},\; \rk(x)\in Q\})\;,\]
and by Claim~1.2 
   \[FC(\{x\;\vert\; x \in S,\; \rk(x)\in P\})=FC(\{x\;\vert\; x \in S\cup\{a_\ell', a_{\ell+1}\},\; \rk(x)\in P\})\;.\]
   Hence, we can deduce that 
   \[FC(\{x\;\vert\; x \in S\cup\{a_\ell', a_{\ell+1}\},\; \rk(x)\in P\})\ind_B FC(\{x\;\vert\; x \in S\cup\{a_\ell', a_{\ell+1}\},\; \rk(x)\in Q\})\;,\]
   as required by the definition of rank-cleanness. For the same reason we also have that 
   \[FC(\{x\;\vert\; x \in S\cup\{a_\ell', a_{\ell+1}\}\}
   {)\ind B}\;.\]
   This concludes the argument for Claim~1.3. 
\end{subproof}

We have thus finished the proof of our induction for Case 1.\\

We now consider \textbf{Case~2}. Since $(\Omega, {\cl})$ forms a pregeometry, we know that for all  $a,b\in\Omega$, $a\in{\cl}(b){\setminus{\cl}(\emptyset)}$ if and only if $b\in{\cl}(a){\setminus{\cl}(\emptyset)}$. Thus, from the hypothesis of Case~2, we can deduce that 
\[a_\ell \ind B\cup \{a_j\;|\; j<\ell,\; \iota_j=1\}\;.\]
We will then be able to use higher Neumann  to extend the sequence $(a_0,a_1,a_1', \dots,a_{\ell-1}, a_{\ell-1}', a_\ell)$ to a rank-{clean} sequence $(a_0,a_1,a_1', \dots,a_{\ell}, a_{\ell}', a_{\ell+1})$ for $u$.

Firstly, set $a_\ell'':=\delta a_\ell$ for an arbitrary witness $\delta$ of \begin{equation}\label{eq:extend2}
(\bar{b}_1,a_{m_1},\dots, a_{m_s}) 
 \equiv (\bar{b}_1',a'_{m_1},\dots, a'_{m_s}) 
 \;.  
\end{equation}
Now, by the above independence condition, we have that $a_\ell''\ind (\overline{b}'_1, a'_{m_1},\dots, a'_{m_s})$. In particular, {setting $S:=\{a_0,a_1,a_1', \dots,a_{\ell-1}, a_{\ell-1}', a_\ell\}$},  higher Neumann gives us 
\begin{equation}\label{eq:homogeneity}
    a_\ell'\equiv_{(\overline{b}'_1, a'_{m_1},\dots, a'_{m_s})} a_\ell''
\end{equation}
such that
\[ {F(a_\ell', \ell)\ind B\cup \bigcup_{a_j^\epsilon\in {S}}F(a_j^\epsilon, j)\;;}\]
{and so, taking the 
{closure} on the right-hand side,}
\begin{equation}\label{eq:wideness consequence}
    F(a_\ell', \ell)\ind B\cup FC({S})\;.
\end{equation}
From (\ref{eq:homogeneity}), 
\[(\bar{b}_1,a_{m_1},\dots, a_{m_s}, a_\ell)\equiv (\bar{b}_1',a'_{m_1},\dots, a'_{m_s}, a'_\ell)\;.\]
{To conclude our proof, we now show that from (\ref{eq:wideness consequence}) and Lemma~\ref{lem:modularmoving} we can prove all conditions showing that $(a_0,a_1,a_1', \dots,a_\ell,a_\ell', a_{\ell+1})$, which is a sequence for $u$ with initial condition $B$ and initial value $a_0$ by construction, is rank-clean. We have $a_0\notin{\cl}(\emptyset)$ since $(a_0)$ is rank-clean by assumption. Next, take $(P, Q)$ to be a partition of the set of freeness ranks occurring in 
the sequence. Suppose without loss of generality that $\rk(a_\ell')\in P$.  By the inductive hypothesis that $(a_0,a_1,a_1', \dots,a_{\ell-1},a_{\ell-1}', a_\ell)$ is rank-clean,}
\[FC(\{x\;\vert\; x\in S,\; \rk(x)\in P\})\ind_B FC(\{x\;\vert\; x\in S,\; \rk(x)\in Q\})\;.\]
From (\ref{eq:wideness consequence}),  base monotonicity, and monotonicity,
\[ F(a_\ell', \ell)\ind_B FC(\{x\;\vert\; x\in S,\; \rk(x)\in P\})\cup FC(\{x\;\vert\; x\in S,\; \rk(x)\in Q\})\;.\]
But now, from Lemma~\ref{lem:modularmoving} {(applied to the relativization ${\cl}_B$, which again yields a modular pregeometry)} and the above two equations, we obtain 
\[ F(a_\ell', \ell)\cup FC(\{x\;\vert\; x\in S,\; \rk(x)\in P\})\ind_B  FC(\{x\;\vert\; x\in S,\; \rk(x)\in Q\})\;.\]
Taking 
{closures}, and noting that $\rk(a_{\ell+1})\in P$ and $F(a_{\ell+1}, \ell+1)\subseteq F(a'_\ell, \ell)$, 
\[ FC(\{x\;\vert\; x\in S\cup\{a_\ell',a_{\ell+1}\},\;  \rk(x)\in P\})\ind_B  FC(\{x\;\vert\; x\in S\cup\{a_\ell',a_{\ell+1}\},\;  \rk(x)\in Q\})\;,\]
as desired. 
An analogous computation also yields
\[FC(\{x\;\vert\; x\in S\cup\{a_\ell',a_{\ell+1}\}\})\ind B\;,\]
thus concluding the proof.
\end{proof}

\begin{theorem}\label{theorem:main} {Let $G$ be a group, and let $w$ be a 
regular word with constants in $G$. 
{Suppose that $G$ has a group action 
$G\acts\Omega$ where $\Omega$ is equipped with a closure operator $\cl$ such that $(\Omega, {\cl})$ forms a $G$-invariant  infinite-dimensional higher Neumann modular pregeometry. 
 Then the set}
\begin{equation}\label{eq:nowheredense}
    \{(g_1, \dots, g_r)\in G^r\;\vert\; w(g_1, \dots, g_r){\neq}  1\}
\end{equation}
  is 
  {dense} in $G^r$ with respect to the topology of pointwise convergence induced by the action on the discrete set $\Omega$. In particular, all mixed identities of $G$ are then singular.}
\end{theorem}
\begin{proof} 
 {Let 
\[\mathcal{U}=\prod_{i=1}^r G_{(\overline{b}_i, \overline{b'_i})}\]
be an arbitrary non-empty basic open set of $G^{{r}}$. We set $B:=((\overline{b}_1, \overline{b}'_1), \dots, (\overline{b}_r, \overline{b}_r'))$. 
 By Lemmas~\ref{lem:basecase} and~\ref{lem:hardinduction}, there is a rank-clean sequence for $w$ with initial condition $B$. Since $w$ is regular, $\rk(a_{n+1})\neq\rk(a_0)$, and so by rank-cleanness, $a_{n+1}\ind a_0$, and in particular $a_{n+1}\neq a_0$. By Lemma~\ref{lem:sequences}, there exists $(g_1, \dots, g_r)\in\mathcal{U}$ such that  $w(g_1, \dots, g_r)(a_0)=a_{n+1}$. Clearly,  $w(g_1, \dots, g_r)\neq 1$, concluding the proof.} 
\end{proof}

\begin{remark} {We note that, having fixed the size of the initial condition and the length of a word $w$, there is a finite upper bound on the sizes of the sets $\Sigma$, $B$, and $C$ for which we require $G\acts\Omega$ to satisfy the higher Neumann condition (as given in Definition~\ref{def:higherneumann}) in  Lemmas~\ref{lem:basecase} and~\ref{lem:hardinduction},
the only {places where}
we appeal to it. Bounding the sizes of these sets in the definition of the condition allows for the $G$-invariant closure operator $\cl$ to be finer than $\acl$, which opens up the possibility of applications to finitary settings. In particular,
 one could thereby quite possibly obtain 
 lower bounds on the shortest regular mixed identity in a given group (or class of groups) (cf.~\cite{bradford2023non, schneider2024word}).}
\end{remark}

%% file: Sections/Examples.tex
\section{Higher Neumann modular pregeometries: superspaciousness}\label{sect:superspacious}

In this section, we show that group actions  $G\acts\Omega$ with no algebraicity,  as well as  the general linear and projective groups of infinite dimension,  
 {yield} infinite-dimensional higher Neumann modular pregeometries. More generally, we isolate the property of superspaciousness which can easily be verified in these cases and which implies higher Neumann.



\begin{definition} 
Let $G\acts\Omega$ be a group action.  Suppose that  $(\Omega, {\cl})$ forms a $G$-invariant 
pregeometry. 
 We say that $G\acts\Omega$ is \textbf{superspacious} if for any {finite} 
set $B\subseteq\Omega$ and any 
{$a\not\in{\cl}(B)$}
the $G_B$-orbit of $a$ intersects every closed subset of $\Omega$ 
 of finite codimension.
\end{definition}



\begin{fact}\label{fact:superspacious} The following are examples of superspacious group actions  $G\acts\Omega$ such that $(\Omega, {\cl})$ are $G$-invariant modular pregeometries: 
\begin{enumerate}
    \item any $G\acts\Omega$ where $\acl$ is locally finite and  $(\Omega, \acl)$ is an infinite-dimensional degenerate pregeometry. This includes all group actions with no algebraicity;
    \item any $G\acts\Omega$ such that $(\Omega, {\cl})$ forms a {$G$-invariant} infinite-dimensional modular pregeometry 
     which is \textbf{homogeneous}\footnote{{We note that this use of the word ``homogeneous''  from the literature on pregeometries has little to do with homogeneous structures as mentioned in the introduction (cf.~\cite{homogeneous}). In the standard definition (see, for example~\cite[$\S 2$ Definition 1.1]{pillay1996geometric}), $B\subseteq\Omega$ is allowed to also be infinite (though we do not need this).}} in the following sense: for any 
    {finite} subset $B\subseteq \Omega$ and $a,a'\in \Omega\setminus {\cl(B)}$ we have $a\equiv_B a'$. This includes $\GL(\kappa, K)$ and $\PGL(\kappa, K)$ when $\kappa$ is an infinite cardinal and $K$ is a field in their action on their corresponding vector space and projective space {with linear closure and projective closure, respectively.}
\end{enumerate}
\end{fact}
\begin{proof} 
{We begin by proving (1). Firstly, as noted earlier, if $\acl$ is a degenerate pregeometry, then it is modular.
Take $B\subseteq\Omega$ 
finite, $a\notin \acl(B)$, and let $X$ be the $G_B$-orbit of $a$. By local finiteness {of $\acl$}, $\dim(X/B)$ is infinite. Let $Y$ be a closed set of finite codimension. Then there exists a finite $F\subseteq\Omega$ such that $\acl(Y\cup F)=\Omega$. Since $\dim(X)$ is infinite, there is some $a'\in X$ not contained in $\acl(F)$. Since $\acl$ is degenerate, $a'\in \acl(Y\cup F)=\Omega$ now implies $a'\in Y$, so $Y\cap X\neq\emptyset$ as desired.}

For (2): 
{given a finite set $B\subseteq \Omega$ and $a\notin{\cl}(B)$, by homogeneity  we have that the $G_B$-orbit $X$ of $a$  is all of $\Omega\setminus {\cl}(B)$, and so clearly $X$ intersects every subset of $\Omega$ of finite codimension.} 
\end{proof}


\begin{lemma}\label{lem:nirvana}
    Let $G\acts\Omega$. Suppose that $(\Omega, \cl)$ forms a modular pregeometry and that $G\acts\Omega$ is superspacious. Then, $(\Omega, {\cl})$ is higher Neumann.
\end{lemma}
\begin{proof} Let $B, C\in[\Omega]^{<\omega}$, {$a\ind B$, and $\Sigma$} 
be as in the definition of higher Neumann and let $X$ denote the $G_B$-orbit of $a$. {If $a\in {\cl}(\emptyset)$, then the conclusion of the higher Neumann condition is {automatically} 
true, so assume this is not the case. Then, {since $a \ind B$,} we have $a\notin{\cl}(B)$.} Extend a basis for ${\cl}(C)$ to a basis for $\Omega$ and let $D$ be the closure of the elements in this basis which are not in ${\cl}(C)$. By construction, $D$ has finite codimension, and by Fact~\ref{fact:indpregeom} we have $D\ind C$. 
By modularity, the intersection of finitely many sets of finite codimension has finite codimension (Fact~\ref{fact:codimension}). {Since $\cl$ is $G$-invariant, $\gamma^{-1} D$ is closed and has the same codimension as $D$ for each $\gamma\in\Sigma$.} Hence, 
\[E:=\bigcap_{\gamma\in\Sigma} \gamma^{-1} D\]
has finite codimension. By  superspaciousness, $E\cap X\neq\emptyset$. Then, picking any $a'\in E\cap X$, 
\[
\Sigma \cdot a'\subseteq D\;,\]
and since $D$ is 
{closed,} 
\[{\cl}(\Sigma\cdot  a')\subseteq D\;.\]
 Thus, 
\[
\Sigma \cdot a'\ind C\;,\]
as desired.
\end{proof}

\begin{remark}
    {Projective closure in {a projective space over an} 
     infinite field  is an explicit example where 
    {we require} a different $G$-invariant closure operator than algebraic closure. For $\PGL(\kappa, K)$ where $K$ is infinite, projective closure does not agree with algebraic closure, {and the latter fails to satisfy}} 
{higher Neumann. To see this, for $\kappa$ an infinite cardinal and $K$ an infinite field, let $V$ be the $\kappa$-dimensional vector space over $K$, and let $PV$ be the associated projective space. For $a\in V\setminus\{0\}$, let $[a]$ denote the corresponding element in $PV$, that is, its  equivalence class  with respect to the equivalence relation
on $V\setminus\{0\}$ identifying vectors that are multiples of each other.
 Below, we write $\pcl$ for projective closure, and $\ind^{\acl}$ for independence with respect to $\acl$. Consider (linearly) independent vectors $a_1$ and $a_2$ from $V$, and let $a:=a_1+a_2$ and $B:=\{[a_1], [a_2]\}$. 
  Since the field $K$ is infinite, $\acl(B)=B$, and clearly $\acl(\{[a]\})=\{[a]\}$. Hence,  $[a]\ind^{\acl} B$. Now, consider any non-zero $\lambda\in K$ such that $\lambda^2\neq 1$, and let $\gamma\in\PGL(\kappa, K)$ be any projectivity induced by a collineation sending $a_1$ to $\lambda a_1$ and fixing $a_2$. Then, set $\Sigma:=\{1, \gamma, \gamma^2\}$. Now, for any ${[a']}\equiv_B [a]$ we have that $\Sigma \cdot {[a']}$ consists of three distinct points on the projective line $\pcl(\{[a_1], [a_2]\})$. By the fundamental theorem of projective geometry, any projectivity fixing three points on the same projective line fixes the entire line pointwise. In particular, $B\subseteq \pcl(\{[a_1], [a_2]\})\subseteq {\acl}(\Sigma \cdot {[a']})$, and so $\Sigma\cdot {[a']}\nind^{\acl} B$. Setting $C:=B$ in the definition of higher Neumann, we have just shown that this property is not satisfied by $\PGL(\kappa, K)$  with respect to $\acl$.} 
\end{remark}

\begin{corollary}\label{cor:main} Let $G$ be any of the following groups:
\begin{itemize}
    \item a group {with an action} $G\acts\Omega$ such that $\acl$ is locally finite and  $(\Omega, \acl)$ is an infinite-dimensional degenerate pregeometry.
   Any action $G\acts\Omega$ with no algebraicity is of this kind;
    \item $\GL(\kappa, K)$ or $\PGL(\kappa, K)$ where $\kappa$ is an infinite cardinal and $K$ is a field. 
\end{itemize}
Let $w$ be a 
regular word with constants in $G$. Then, the set
\[\{(g_1, \dots, g_r)\in G^r\;\vert\; w(g_1, \dots, g_r)\neq 1\}\]
  is dense in $G^r$ {with respect to the topology of pointwise convergence induced by the action on the discrete set $\Omega$}. 
In particular, all mixed identities of $G$ are singular.
\end{corollary}
\begin{proof} The corollary is just a consequence of putting together Fact~\ref{fact:superspacious}, Lemma~\ref{lem:nirvana}, and Theorem~\ref{theorem:main}.
\end{proof}

%% file: Sections/Microsupported.tex
\section{Micro-supported actions}\label{sec:micro-supported}

We mentioned that our main result, Theorem~\ref{theorem:main}, can be viewed as a vast generalization of Ab\'{e}rt's observation  that  group actions  without algebraicity are lawless; our techniques do not reduce the problem to the case without constants, but prove the latter as a corollary. 
Moving to a topological setting, we now show a result of similar spirit:  all mixed identities are singular in groups with micro-supported actions, again generalizing the lawlessness of such groups due to  Nekrashevych~\cite[Theorem 2.2.5]{nekrashevych2022groups}. In this case we can, without too much effort, reduce the problem to identities without constants.

\begin{definition}\label{def:microsupported}
   Let $G$ be a group, let $\spOm=(\Omega;\tau)$ be a Hausdorff topological space,  and let $G\acts \Omega$ be a group action  by homeomorphisms of $\spOm$. 
We say that $G\acts \Omega$ is \textbf{micro-supported} if $G_{\Omega\setminus U}\neq\{1\}$
for every non-empty open set $U\subseteq \Omega$.
\end{definition}

Note that Hausdorff spaces $\spOm$  admitting a micro-supported action $G\acts \Omega$ have  no isolated points since for any  $a\in \Omega$, $G_{\Omega\setminus\{a\}}=\{1\}$ (and hence $\{a\}$ cannot be open). In particular, $\Omega$ 
must then be  infinite.

\begin{theorem}[Theorem 2.2.5 in~\cite{nekrashevych2022groups}]\label{thm:microlawless}
    Let $G$ be a group, and suppose that $G$ has a micro-supported  group action $G\acts \Omega$ on a  Hausdorff topological space
    $\spOm$. Then, $G$ is lawless. 
\end{theorem}

In our proof we will require the following well-known fact (cf.~\cite[Corollary 2.2.6]{nekrashevych2022groups}). 

\begin{fact}\label{fact:local-lawlessness} Let $\spOm$ be a Hausdorff topological space and let $G\acts \Omega$ be a micro-supported group action. Let $V\subseteq \Omega$ be non-empty and open. Then,  $G_{\Omega\setminus V}\acts V$ is micro-supported.
\end{fact}
\begin{proof}
If
$W\subseteq V$ is non-empty and open in $V$, then $W$ is open in $\spOm$. Hence, $G_{\Omega\setminus W}\neq \{1\}$. {Moreover, the restricted action is faithful: if $g\in G_{\Omega\setminus V}$ fixes $V$ pointwise, then $g=1$.}
%
\end{proof}

The following is a standard fact about micro-supported actions, which follows, for example, from~\cite[Lemma 6.5.(2)]{kim2021structure}.

\begin{lemma}\label{lem:finite-germ-separation} Let $\spOm$ be a Hausdorff topological space, and let $G\acts \Omega$ be a group action by homeomorphisms of $\spOm$.  For every finite
set $\Sigma\subseteq G$ and every non-empty open set {$O\subseteq \Omega$},
there is a
non-empty open set $V\subseteq {O}$
 such that, for all $\alpha,
\beta\in\Sigma$,
\begin{itemize}
    \item EITHER $\alpha\mathord\upharpoonright_V=\beta\mathord\upharpoonright_V$; 
    \item OR $\alpha(V)\cap \beta(V)=\emptyset$.
\end{itemize}
\end{lemma}

\begin{proof}
By iterating the argument, it is sufficient to show the statement for $|\Sigma|=2$. Writing $\Sigma=\{\alpha,\beta\}$ and setting $\gamma:=\beta^{-1}\alpha$, we have to find a non-empty open set $V\subseteq O$ such that $\gamma\mathord\upharpoonright_V=\operatorname{id}_V$ or $V\cap \gamma(V)=\emptyset$. {If $\gamma$ {fixes}
all elements of $O$, then we can set $V:=O$. Otherwise, for any $a\in O$ with $\gamma(a)\neq a$, using continuity and Hausdorffness we can pick an open neighbourhood $V$ of $a$ in $O$ such that $V\cap \gamma(V)=\emptyset$, proving the statement.}
\end{proof}

{The following lemma allows us to simplify our verification that non-solution sets of equations with regular words are dense in $G^r$.}

\begin{lemma}\label{lemma:slightly technical} 
Let $G\acts\Omega$ be any group action. 
Suppose that for any regular word ${w}\in G* F_r$ and for any finite $B\subseteq \Omega$, there are $p_1, \dots, p_r\in G_B$ such that ${w}(p_1, \dots, p_r)\neq 1$. Then, for any regular word $w\in G* F_r$, the set 
\begin{equation}\label{eq:thesolutionset}
   \{ (g_1, \dots, g_r)\in G^r\ \vert\ w(g_1, \dots, g_r){\neq} 1\}\; 
\end{equation}
{is dense in $G^r$ with respect to $\pw$.}
\end{lemma}
\begin{proof} {Let $\mathcal{U}$ be a non-empty basic open subset of $G^r$. We may write}
\[
  \mathcal{U}
  = \prod_{i=1}^r G_{(\overline{b}_i, \overline{b}'_i)},\]
{where $\overline{b}_i$ and $\overline{b}_i'$ are tuples from $\Omega$.} {To prove density,} it suffices to find $(g_1,\ldots,g_r)\in \mathcal U$ with $w(g_1,\ldots,g_r)\neq 1$. Choose
$\mathbf f=(f_1,\ldots,f_r)\in\mathcal U$, and let $B\subseteq \Omega$ be the
finite set consisting of all entries of the tuples
$\bar b_1,\ldots,\bar b_r$. Let $w_{\mathbf f}$ be the word obtained by {replacing} each occurrence of $x_i$ by $f_ix_i$, for all {$i\in\{1, \dots, r\}$.} Clearly, $w_{\mathbf f}$ is regular, and for every
$\mathbf p=(p_1,\ldots,p_r)\in G^r$, 
\begin{equation}\label{eq:substitution-evaluation}
  w_{\mathbf f}(p_1,\ldots,p_r)
  =w(f_1p_1,\ldots,f_rp_r).
\end{equation}
By assumption, there are
$p_1,\ldots,p_r\in G_B$ such that $w_{\mathbf f}(p_1,\ldots,p_r)\neq 1$.
For  $1\leq i\leq r$, set $g_i:=f_i p_i$. Since every $p_i$ fixes $B$ pointwise, $g_i$ agrees with $f_i$ on
$\bar b_i$, and so $(g_1,\ldots,g_r)\in\mathcal U$. By \eqref{eq:substitution-evaluation}, $w(g_1, \dots, g_r)\neq1$.
\end{proof}

\begin{theorem}\label{thm:micro-supported-mixed-identities} Let $G$ be a group, and let $w\in G*F_r$ be a regular word with constants in $G$. If $G$ has a micro-supported group action on a Hausdorff topological space $\spOm$, then the set
\[
  \{(g_1,\ldots,g_r)\in G^r\ \vert \ w(g_1,\ldots,g_r) {\neq} 1\}
\]
is dense in $G^r$ with respect to the topology of pointwise convergence 
 induced by the action {on $\Omega$ considered as a discrete set}. 
In particular, every mixed identity of $G$ is then  singular.
\end{theorem}
\begin{proof}
Let $w$ be written  as in (\ref{eq:word}).  We only need to verify the condition of Lemma~\ref{lemma:slightly technical}, i.e.~for any  $B\subseteq \Omega$ finite we will find $p_1, \dots, p_r\in G_B$ such that ${w}(p_1, \dots, p_r)\neq 1$. For every $k\in [n]$, let ${w}_k$ be the 
{initial segment (from the right)} 
of ${w}$ of length $2k+1$ and $u_k:=\rk({w}_k)$ (which is a word without constants in $F_r$). 
Since ${w}_n=w$ is regular, $\rk({w}_n)=u_n\neq 1$. For every $k\in [n]$, set
\[
  \xi_k:=\gamma_k\gamma_{k-1}\cdots\gamma_0\;.
\]

 Since $
C:=\bigcup_{k\in [n]}\xi_k^{-1}(B)$
is finite and $\spOm$ is Hausdorff, the set $O:=\Omega\setminus C$ is non-empty and open. Applying 
Lemma~\ref{lem:finite-germ-separation} to $ \Sigma:=\{1,\xi_0,\ldots,\xi_n\}$
and to $O$, we obtain a non-empty open set $V\subseteq O$ such that any two
members of $\Sigma$ either agree on $V$ or have disjoint images of $V$. 
Since $V\subseteq \Omega\setminus C$, for all $k\in [n]$ we have that
\begin{equation}\label{eq:cells-avoid-B}
  \xi_k(V)\cap B=\emptyset\;.
\end{equation}
By Fact~\ref{fact:local-lawlessness}, $G_{\Omega\setminus V}\acts V$ is micro-supported, and hence Theorem~\ref{thm:microlawless} assures there exist 
$h_1,\ldots,h_r\in G_{\Omega\setminus V}$ such that $u_n(h_1,\ldots,h_r)\neq 1$.


 Let $D\subseteq [n]$ be such that for each $k\in [n]$ 
 there is a unique $j\in D$ such that $\xi_k(V)=\xi_j(V)$. For each 
$i\in\{1,\ldots,r\}$ and each $j\in D$, let
\[
  q_{i,j}:=\xi_{j}h_i \xi_{j}^{-1}\; .
\]
Since $h_i$ fixes $\Omega\setminus V$ pointwise,  $q_{i,j}$ fixes
$\Omega\setminus \xi_j(V)$ pointwise. Hence, if $j,k\in D$ are distinct, then $q_{i,j}$ and $q_{i,k}$ commute. 
It follows that for each $i\in\{1, \dots, r\}$ the product 

\begin{equation}\label{eq:patched-variable}
  p_i:=\prod_{j\in {D}}
  q_{i,j} 
\end{equation}
is independent of the order of its factors, and 
\begin{equation}\label{eq:restriction-to-cell}
  p_i\mathord\upharpoonright_{\xi_j(V)}
  =q_{i,j}\mathord\upharpoonright_{\xi_j(V)}
\end{equation}
for all 
 $j\in D$.
Moreover, by \eqref{eq:cells-avoid-B}, every $\xi_j(V)$ is disjoint
from $B$, and so, for each $i\in\{1, \dots, r\}$ we have  $p_i\in G_B$. Finally, it follows straightforwardly from the definition that {for any  ${\epsilon}\in\{-1,1\}$, any $i\in\{1,\ldots,r\}$, and any $k\in[n]$}
\begin{equation}\label{eq:coordinate-action}
  (\xi_k^{-1}\ p_i^{{{\epsilon}}}\ \xi_k)\mathord\upharpoonright_V
  =h_i^{{{\epsilon}}} \mathord\upharpoonright_V\; ;
\end{equation}
{this follows from the definition of $p_i$ for $k\in D$, and for all other $k\in [n]$ since $\xi_k\mathord\upharpoonright_V=\xi_j\mathord\upharpoonright_V$ for some $j\in D$.}

We now show that for every $c_0\in V$ and every $k\in [n]$,
\begin{equation}\label{eq:inductive-evaluation}
  {w}_k(p_1,\ldots,p_r)(c_0)
  =\xi_k u_k(h_1,\ldots,h_r)(c_0).
\end{equation}
We use induction on $k$. The case $k=0$
is immediate since ${w}_0=\gamma_0=\gamma_0 u_0$. Suppose \eqref{eq:inductive-evaluation} holds for $k-1$, where $1\leq k\leq n$. Set 
\[
  c_{k-1}:=u_{k-1}(h_1,\ldots,h_r)(c_0).
\]
 Then $c_{k-1}\in V$ since $h_i\in G_{\Omega\setminus V}$ for all $i\in\{1,\ldots,r\}$. 
By the inductive hypothesis and~\eqref{eq:coordinate-action},
\[
\begin{aligned}
  {w}_k(p_1,\ldots,p_r)(c_0)
  &=\gamma_k\; p_{\iota_k}^{{\epsilon}_k}\;  w_{k-1}(p_1,\ldots,p_r)(c_0)\\
  &=\gamma_k\;  p_{\iota_k}^{{\epsilon}_k}\;  \xi_{k-1}\;u_{k-1}(h_1,\ldots,h_r)(c_0)\\
  &=\gamma_k\;  p_{\iota_k}^{{\epsilon}_k}\;  \xi_{k-1}(c_{k-1})\\
  &=\gamma_k\; \xi_{k-1}
    h_{\iota_k}^{{\epsilon}_k}(c_{k-1})\\
    &=\xi_{k}\; 
    h_{\iota_k}^{{\epsilon}_k}(c_{k-1})\\
  &=\xi_k\; u_k(h_1,\ldots,h_r)(c_0),
\end{aligned}
\]
proving~\eqref{eq:inductive-evaluation}. 

We conclude our proof by showing that there is some $c\in V$ such that $$
w(p_1,\ldots,p_r){(c)}=\xi_n u_n(h_1, \dots, h_r)(c)\neq c\; .$$
{By construction of $V$, since $1,\xi_n\in\Sigma$, either $\xi_n(V)\cap V=\emptyset$ or
$\xi_n\mathord\upharpoonright_V=\operatorname{id}_V$. In the first case, pick any $c\in V$. Since $u_n(h_1, \dots, h_r)(c)\in V$, we have that $\xi_nu_n(h_1,\ldots,h_r)(c)\neq c$ as desired. Consider the second case. We have chosen $h_1, \dots, h_r$  so that  $u_n(h_1, \dots, h_r)$ is not the identity on $V$. So, there is some $c\in V$ such that $u_n(h_1, \dots, h_r)(c)\neq c$. Since $\xi_n$ fixes $V$ pointwise, we also have that $\xi_n u_n(h_1, \dots, h_r)(c)\neq c$. This concludes the proof of the theorem.}
\end{proof}
\begin{remark}
    {Inspecting the proof of Theorem~\ref{thm:micro-supported-mixed-identities}, one sees that given any set $N$ belonging to a $G$-invariant ideal of non-dense subsets of $\spOm$ (such as a nowhere dense set), we can even find a tuple $(g_1,\ldots,g_r)\in (G_{N})^r$ violating the equation $w(x_1,\ldots,x_r)=1$.}
\end{remark}

\begin{remark} {We note that several groups studied in topological dynamics have micro-supported actions, and so Theorem~\ref{mainthm:micro} applies to them. Firstly, Rubin actions in the sense of~\cite{belk2025short} are micro-supported, and so several groups of {homeomorphisms, diffeomorphisms,}
and Thompson-like groups have micro-supported actions. Indeed, the notion of micro-supported actions arises from Rubin's work on reconstructing spaces from their groups of symmetries~\cite{rubin1989reconstruction}. In particular, a group has a micro-supported action if and only if it has a locally moving action in Rubin's sense~\cite{rubin1996locally}. Being locally moving is defined in terms of an action of $G$ on a complete atomless Boolean algebra $\mathbf{B}$~\cite[Definition 2.4]{rubin1996locally}. However, it is easy to see that the action of $G$ on the associated Stone space $S(\mathbf{B})$ is micro-supported and that micro-supported actions induce locally moving actions~\cite[Proposition 1.8 (c)]{rubin1996locally}. Thus, the large literature on locally moving groups provides further examples of groups to which Theorem~\ref{mainthm:micro} applies. This includes, for example, the automorphism groups of measure algebras (cf.~\cite{maharam1942homogeneous} and~\cite[p.127]{rubin1996locally}), or the subgroups of the automorphism groups of linear and circular orders considered in~\cite{mccleary2005locally}.}  
\end{remark}